\documentclass[10pt,leqno,a4paper]{article}
\usepackage{amsmath} 
\usepackage{amsthm}
\usepackage{amssymb}
\usepackage{tikz}
\usepackage{tikz-cd}
\usetikzlibrary{matrix,arrows,decorations.pathmorphing}
\usepackage{graphicx}
\usepackage[english]{babel}
\usepackage[totalheight=22 true cm, totalwidth=14 true cm]{geometry}
\usepackage{setspace}
\usepackage{mathrsfs}
\usepackage[all]{xy}

\newtheorem{thm}{Theorem}[section]
\newtheorem{cor}[thm]{Corollary}
\newtheorem{lemma}[thm]{Lemma}
\newtheorem{sublemma}[thm]{Sublemma}
\newtheorem{prop}[thm]{Proposition}

\newtheorem{conv}[thm]{Convention}

\newtheorem{defn}[thm]{Definition}

\theoremstyle{remark}

\theoremstyle{definition}

\newtheorem{rmk}[thm]{Remark}

\numberwithin{equation}{thm}
\def\beq{\begin{equation}}
\def\eeq{\end{equation}}
\def\beqa{\begin{equation*}}
\def\eeqa{\end{equation*}}
\def\ben{\begin{enumerate}}
\def\een{\end{enumerate}}
\def\besp{\begin{split}}
\def\eesp{\end{split}}

\def\SUMU{ \mathop{{\bigoplus}^u}\limits}

\def\PRODSCU#1{\mathop{{\mathop{\prod_{\substack #1}}\limits}^{\square,u}}\limits}

\def\limPROu{ \mathop{\underleftarrow{\rm lim}}\limits}

\def\limINDu{ \mathop{\underrightarrow{\rm lim}}\limits} 
\def\limINDU{ \mathop{\underrightarrow{\rm lim}^{ u}}\limits}

\def\Ab{{\cA b}}
\def\AP{{\cA\cP}}
\def\APH{{\cA\cP\cH}}
\def\Mod{{\cM od}}

\def\ACLM{{\cA\cC\cL\cM}}

\def\LMu{{\cC\cL\cM}^{u}}
\def\ACLMu{{\cA\cC\cL\cM}^u}
\def\LC{{\cL\cC}}
\def\CLC{{\cC\cL\cC}}

\def\crash#1{}
\def\N{{\mathbb N}}

\def\Z{{\mathbb Z}}
\def\0{{\mathbb O}}

\def\P{{\mathbb P}}
\def\Q{{\mathbb Q}}
\def\R{{\mathbb R}}
\def\C{{\mathbb C}}

\def\A{{\mathbb A}}
\def\E{{\mathbb E}}
\def\F{{\mathbb F}}

\def\G{{\mathbb G}}

\def\T{{\mathbb T}}

\def\1{{\mathbf 1}}

\def\l{\left}
\def\r{\right}
\def\[[{\l[\l[}
\def\]]{\r]\r]}

\def\cf{\emph{cf.}\,}
\def\ie{\emph{i.e.}\,}
\def\lc{\emph{loc.cit.\,}}

\def\cA{{\mathcal A}}

\def\cC{{\mathcal C}}

\def\cM{{\mathcal M}}

\def\cO{{\mathcal O}}
\def\cH{{\mathcal H}}

\def\cL{{\mathcal L}}

\def\cP{{\mathcal P}}

\def\g"{``}
\def\g'{`}

\def\sC{{\mathscr C}}

\def\sE{{\mathscr E}}
\def\sF{{\mathscr F}}
\def\sH{{\mathscr H}}

\def\sL{{\mathscr L}}
\def\sM{{\mathscr M}}

\def\sP{{\mathscr P}}

\def\sT{{\mathscr T}}

\def\what{\widehat}

\def\veps{\varepsilon}

\def\unif{{\rm unif}}

\def\ol{\overline}
\def\iso{\xrightarrow{\ \sim\ }}
\def\map#1{\xrightarrow{\ #1\ }}

\def\limind{\mathop{\lim\limits_{\displaystyle\rightarrow}}}

\newcommand{\dspl}{\displaystyle}

\def\wt{\what{\otimes}}

\def\bd{{\rm bd}}

\author{Francesco Baldassarri  
\thanks{Universit\`{a} di Padova,
Dipartimento di Matematica, Via Trieste, 63, 35121 Padova, Italy.
 } 
 \\[10pt]
{with an Appendix by Maurizio Candilera}\;$^{\ast}$
 }

\title{A $p$-adically entire function with integral values on $\Q_p$\\
and entire liftings of the $p$-divisible group $\Q_p/\Z_p$
}

\begin{document}
 \date{\today}

\maketitle
\begin{abstract}
We  give a  self-contained proof of the fact, discovered in \cite{tesi} and proven in \cite{psi} with the methods of \cite{MA}, that, for any prime number $p$, there exists a  power series
$$\Psi= \Psi_p(T) \in T + T^2\Z[[T]] 
$$
 which trivializes the addition law of the formal group of Witt covectors 
 \cite{MA}, \cite[II.4]{fontaine}, is $p$-adically entire and assumes values in $\Z_p$ all over $\Q_p$. 
 We actually generalize, following a suggestion of M. Candilera, the previous facts to any fixed unramified extension $\Q_q$ of $\Q_p$ of degree $f$, where $q = p^f$. We show that $\Psi = \Psi_q$ provides a quasi-finite  covering of the Berkovich affine line $\A^1_{\Q_p}$ by itself. We prove in section~\ref{newton} new strong estimates for the growth of $\Psi$,  in view of the 
 application \cite{perf_fourier} to $p$-adic Fourier expansions on $\Q_p$. We refer to \cite{perf_fourier}  for the proof of a technical corollary (Proposition~\ref{estadm}) which we apply here to locate 
the zeros of $\Psi$ and to obtain its  product expansion (Corollary~\ref{invfunct}). 
 \par
 We reconcile the present discussion (for $q =p$) with the formal group proof given in \cite{psi} which takes place in the Fr\'echet algebra $\Q_p\{x\}$ of the analytic additive group $\G_{a,\Q_p}$ over $\Q_p$. 
We show that, for any $\lambda \in \Q_p^\times$,
the closure $\sE_\lambda^\circ$ of $\Z_p[\Psi(p^ix/\lambda)\,|\,i=0,1,\dots]$ in $\Q_p\{x\}$ 
 is a  Hopf algebra object in the category of 
  Fr\'echet $\Z_p$-algebras.   
\par    
  The special fiber of $\sE_\lambda^\circ$ is 
 the affine algebra of  
 the $p$-divisible group $\Q_p/p \lambda \Z_p$ over $\F_p$, while $\sE_\lambda^\circ [1/p]$ is dense in $\Q_p\{x\}$.
 \par  
From  $\Z_p[\Psi(\lambda x)\,|\,\lambda \in \Q_p^\times]$  we construct  a $p$-adic   analog 
$\AP_{\Q_p}(\Sigma_\rho)$  of the algebra  of  
Dirichlet series holomorphic in a strip $(-\rho, \rho) \times i \R \subset \C$. 
We  start developing this analogy. It turns out that  the Banach algebra of   almost periodic functions on $\Q_p$ 
identifies with the topological ring of germs of  holomorphic almost periodic functions on  strips around $\Q_p$.
  \end{abstract}

\tableofcontents
\bigskip
\setcounter{section}{-1}

\section{Introduction}
\subsection{Foreword} An unfortunate feature of $p$-adic numbers is that there exists no   character
 $$\psi: (\Q_p,+) \to (\C_p^\times,\cdot)\;\;,\;\;\psi \neq 1
 $$
 which extends to an entire function $\C_p \to \C_p$. In fact, let $\pi_p \in \C_p^{\circ \circ}$ be such that 
 the radius of convergence of $\exp(\pi_p x)$ equals 1, so that $\exp$ and $\log$ establish an isomorphism
 $$
 (\pi_p \, \C_p^{\circ \circ},+) \iso ( \exp (\pi_p \, \C_p^{\circ \circ}),\cdot) \;\mbox{(} \subset (1 + \C_p^{\circ \circ}, \cdot)
 \mbox{)}\;.
 $$ 
 Now, assume  a $\psi$ as above exists, and let $n$  be a positive integer such  that 
 $\psi (p^n) \in \exp (\pi_p \, \C_p^{\circ \circ})$ so that $\psi$ restricts to a character 
 $\psi : (p^n \Z_p,+) \to ( \exp (\pi_p \, \C_p^{\circ \circ}),\cdot)$. 
 Let $a := \log ( \psi (p^n))$. Then, for any $x \in \Z$,   $\psi(p^nx) = \psi(p^n)^x = \exp (a x)$.   
 But $x \mapsto \exp (ax)$ has a finite radius of convergence.
 \par
    We partially remedy to the previous inconvenience by 
 showing the existence, for any $\lambda \in \Q_p^\times$, of a representable formal group functor 
\beq \label{tilambda}
\E_\lambda : \ACLMu_{\Z_p} \to \Ab
\eeq
(see section~\ref{lincat} in Appendix A, for notation) whose generic (resp. special) fiber is the $\Q_p$-analytic group $\G_a$ 
 (resp. the constant $p$-divisible group $\Q_p/\lambda
  \Z_p$ over $\F_p$). 
 The idea is the following.  Over the complex numbers the formulas 
 $$
 e^{ i z} =   \cos z + i \sin z \;\;,\;\; e^{ -i z} =   \cos z - i \sin z
 $$
show  that the two (Hopf) algebras    $\Z[i][e^{ i z} ,  e^{ -i z}]$  and $\Z[i][\sin z, \cos z]$ coincide.   The sequence of functions 
\beq \label{coord}
\Psi(x) = \Psi_p(x),\Psi(px), \Psi(p^2x),\dots 
\eeq
plays here the role of the pair $(\cos z , \sin z)$ in that  the $p$-adically entire and integral addition law \eqref{covsum2} holds, and $x$ is a logarithm for that formal group.  So, while it is improper to say that $\Psi$  plays the role of an entire character of $\Q_p$, it is suggestive to consider a suitable $p$-adic completion of the algebra $\Z_p[\Psi(\lambda x)\,|\, \lambda \in \Q_p^\times]$ and to compare it with the classical algebras of Bohr's almost periodic functions $APH_\R$ and   $AP_\R$. We review  for  convenience the classical definitions of real and complex Fourier analysis in  section~\ref{expvstrig} of Appendix B. 
A closer $p$-adic analog of those classical constructions, and a generalization of Amice-Fourier theory to $p$-adic functions on $\Q_p$, will appear in \cite{perf_fourier}.

 \subsection{The function $\Psi$}
In the paper \cite{psi} we introduced, for any prime number $p$, a power series
$$\Psi(T) = \Psi_p(T) = T + \sum_{i=2}^{\infty} a_{i} T^{i}\in \Z[[T]]\;,
$$
which represents an entire $p$-adic analytic function, \ie  is such that
\beq \label{conv} \limsup_{i \to \infty} |a_{i}|_p^{1/i} = 0 \; .
\eeq
 This function has the remarkable property that $\Psi_p(\Q_p) \subset \Z_{p}$ and that, for any $i \in \Z$ and $x \in \Q_p$, if we write $x$ as in \eqref{wittsum}, with $x_i$ defined  by \eqref{wittcomp2}, \eqref{wittcomp3}, then
\beq \label{wittcomp}
x_{-i} = \Psi_p(p^i x)\; \; {\rm mod} \; p \;\; \in \F_p\;.
\eeq
The power series $\Psi(T)$ is defined by the functional relation
\beq  \label{functeq}
\sum_{j=0}^{\infty} p^{-j}\Psi(p^{j} T)^{p^{j}} = T \; .
\eeq
Its inverse function $\beta=\beta_p \in T +T^2\Z[[T]]$ was shown to converge exactly in the region
\beq \label{convregion}
|T|_p < p\;\; \mbox{\ie} \;\;v_p(T)>-1 \; .
\eeq
\endgraf 
One property we had failed to notice in \cite{psi} is the following 
\begin{prop} \label{unifPsi}
The restriction of the function  $\Psi_p$ to a map $\Q_p \to \Z_p$ is uniformly continuous. More precisely, for any $j =0,1,\dots$ and $x \in \Q_p$,
\beq
\label{unifPsi1}
\Psi_p(x + p^j  \C_p^\circ ) \subset \Psi_p(x) + p^j  \C_p^\circ \;.
\eeq
\end{prop}
This is proven in Corollary~\ref{basicest0} below.  See also the more general Theorem~\ref{psibddsuper} whose proof depends on Proposition~\ref{estadm}, proven in \cite{perf_fourier}. 
\subsection{Our previous approach \cite{psi}} \label{previousapproach}
Proofs in \cite{psi} were  based on  Barsotti-Witt algorithms  \cite{MA}.
The most basic notion of topological algebra in \cite{MA} is the one of a \emph{simultaneously admissible} 
 family, indexed by $\alpha \in A$, of sequences $i \mapsto x_{\alpha,-i}$ for $i=0,1,\dots$ in a 
Fr\'echet algebra $R$ over $\Z_p$ (in particular, over $\F_p$)  \cite[Ch.1, \S 1]{MA}.
In case $R$ is a  Fr\'echet algebra  over $\Q_p$ the definition of simultaneous  admissibility is more restrictive, but the 
name used in \lc is the same.   For clarity,  the more restrictive notion will be called here (simultaneous) \emph{$\rm PD$-admissibility}, while the general notion will maintain the name of (simultaneous) \emph{admissibility}. 
\par
 Using the previous refined terminology, our main technical tool in \cite{psi}
was  a criterion \cite[Lemma 1]{psi} of \emph{simultaneous $\rm PD$-admissibility}  
for a family indexed by $\alpha \in A$, of sequences $i \mapsto x_{\alpha,-i}$ for $i=0,1,\dots$ in a 
Fr\'echet $\Q_p$-algebra. 
In
 Barsotti's theory of  $p$-divisible groups one  regards an admissible sequence $i \mapsto x_{-i}$ 
as a \emph{Witt covector} $(\dots,x_{-2},x_{-1},x_0)$ \cite{MA}, \cite{fontaine} with components $x_{-i} \in R$. 
\endgraf
We take here only a short detour on the group functor viewpoint
and refer the reader to \cite{fontaine} for precisions. As abelian group functors on a suitable category of topological $\Z_p$-algebras the
direct limit  ${\rm W}_n \to {\rm W}_{n+1}$ of the Witt vector groups of length $n$ via the \emph{Verschiebung map} 
 $$
V: (x_{-n},\dots, x_{-1},x_0) \to (0,x_{-n},\dots, x_{-1},x_0)  
 $$
 indeed exists.  It is the group functor  $\rm CW$ of \emph{Witt covectors}. 
For a topological $\Z_p$-algebra $R$ on which ${\rm CW}(R)$ is defined, it is convenient to denote an element $x \in {\rm CW}(R)$ by  an inverse sequence 
$$x=(\dots,x_{-2},x_{-1},x_0)$$ 
of elements of $R$, that is a Witt covector with components in $R$. 
Two Witt covectors $x=(\dots,x_{-2},x_{-1},x_0)$ and $y=(\dots,y_{-2},y_{-1},y_0)$ with components $R$ can be summed 
by taking limits of sums of finite  Witt vectors. Namely, 
let
 \beq
 \label{limphi0}
 \varphi_i(X_0,\dots,X_i;Y_0,\dots,Y_i) \in \Z[X_0,\dots,X_i,Y_0,\dots,Y_i]
\eeq
be  the $i$-th (= the last!) entry of the Witt vector $(X_0,\dots,X_i)+(Y_0,\dots,Y_i)$. 
Then, 
$$x + y = z = (\dots,z_{-2},z_{-1},z_0)
$$ 
means that, for any $i=0,-1,\dots$,
\beq\label{limphi}
z_i 
 = \lim_{n \to +\infty} \varphi_n(x_{i-n},x_{i-n +1},\dots,x_i;y_{i-n},y_{i-n+1},\dots,y_i) 
 \eeq
 converges in $R$. 
 The convergence properties on the 
 Witt covectors $x$ and $y$ above for the expressions \eqref{limphi} to converge, are dictated by the following
  \begin{lemma} \label{phiisob} (\cite[Teorema 1.11]{MA}) Notation as  above.
 For $i =0,1,2, \dots$, let us attribute the weight $p^i$ to the variables $X_i, Y_i$. Then, for any $i \geq 0$ the polynomial $\varphi_i$ in \eqref{limphi0} is isobaric
 of weight $p^i$. Moreover, for any $i\geq 1$,
\beq  \label{congphi}  \begin{split}
\varphi_i(X_0,X_1,&\dots,X_i;Y_0,Y_1,\dots,Y_i) - \varphi_{i-1} (X_1,\dots,X_i;Y_1,\dots,Y_i)  \in \\
&
\\ & X_0\,Y_0\, \Z[X_0,X_1,\dots,X_i,Y_0,Y_1,\dots,Y_i] \;.
 \end{split}
 \eeq
\end{lemma}
So, we equip the polynomial ring $\Z[X_0,X_{-1},\dots,X_{-i},\dots;Y_0,Y_{-1},\dots,Y_{-i},\dots]$ with the linear topology defined by the powers of the ideals $I_N:=(X_{-N},X_{-N-1}, \dots;Y_{-N},Y_{-N-1}, \dots)$ and set 
$$
\cP:= \limPROu_{N,M \to +\infty} \Z[X_0,X_{-1},\dots,X_{-i},\dots;Y_0,Y_{-1},\dots,Y_{-i},\dots]/I_N^M \;.
$$
Then, the sequence  
\beq \label{Phidef0} 
i \longmapsto \varphi_i(X_{-i},\dots,X_{-1},X_0;Y_{-i},\dots,Y_{-1},Y_0)
\eeq
converges to  an element
\beq \label{Phidef00} 
\Phi(X_0,X_{-1},\dots,X_{-i},\dots;Y_0,Y_{-1},\dots,Y_{-i},\dots) \in \cP\;.
\eeq
So,  \eqref{limphi} is expressed more compactly as 
\beq\label{limphi1}
z_i =  \Phi(x_i,x_{i-1}, \dots;y_i,y_{i-1}, \dots) 
 \eeq
\begin{rmk}\label{usualwitt} The  projective  limit  
 $${\rm W}_{n+1} \to {\rm W}_n\;\;\;,\;\;\;(x_0,x_1,\dots,x_{n+1}) \mapsto  (x_0,x_1,\dots,x_n) \;,$$
 produces instead the algebraic group $\rm W$  of \emph{Witt vectors}. 
 \end{rmk}
 \endgraf
The approach of Barsotti \cite{MA} is more flexible and easier to apply to analytic categories. 
If $R$ is complete, for two simultaneously admissible 
Witt covectors $x=(\dots,x_{-2},x_{-1},x_0)$ and $y=(\dots,y_{-2},y_{-1},y_0)$ with components $R$ 
the expressions \eqref{limphi1} all converge in $R$ and  define $(\dots,z_{-2},z_{-1},z_0) = z =: x+y$, which is in turn 
simultaneously admissible with $x$ and $y$.
In the $\Q_p$-algebra case  a Witt covector $x=(\dots,x_{-2},x_{-1},x_0)$ has \emph{ghost components} 
$(\dots,x^{(-2)},x^{(-1)},x^{(0)})$ 
defined by 
\beq \label{ghost} 
x^{(i)} = x_i + p^{-1} x_{i-1}^p + p^{-2} x_{i-2}^{p^2} + \dots \;\;,\;\; i = 0,-1,-2,\dots \;.
\eeq
Under very general  assumptions \cite[Teorema 1.11]{MA}, a finite family of  sequences 
$(x_{\alpha,-i})_{i=0,1,\dots}$, for $\alpha \in A$ in a $\Q_p$-Fr\'echet algebra are simulaneously PD-admissible iff the same holds for the family of sequences of ghost components $(x_{\alpha}^{(-i)})_{i=0,1,\dots}$, for $\alpha \in A$. 
Under these assumptions, for simultaneously PD-admissible covectors $x$ and $y$,  $x+y =z$ is equivalent to
\beq \label{ghost2} 
z^{(i)} = x^{(i)} + y^{(i)}  \;\;,\;\; i = 0,-1,-2,\dots \;.
\eeq
In the present case, which coincides with the case treated in  \cite{psi},
the sequences $i \mapsto x_{-i} :=p^i x$ and $i \mapsto y_{-i}:=p^i y$  are simultaneously PD-admissible in the standard $\C_p$-Fr\'echet algebra $\C_p\{x,y\}$ of entire functions on $\C^2_p$ \cite[Lemma 1 and Lemma 3]{psi}. It follows from the relation \eqref{functeq} that 
$i \mapsto x_{-i} := p^i x$, for $i=0,1,2,\dots$  is the sequence of ghost components of $x \mapsto x^{(-i)} := \Psi(p^ix)$. 
Therefore from  \cite[\lc]{MA} we conclude that the two sequences $i \mapsto \Psi(p^{i}x)$ and $i \mapsto \Psi(p^{i}y)$ are  simultaneously admissible in $\C_p\{x,y\}$, as well. Moreover, by  \cite[\lc]{MA} and the definition of the addition law of Witt covectors with coefficients in $\C_p\{x,y\}$, we have
\beq
\label{covsum} \begin{split}
 ( \dots, \Psi(p^{2}(x+y)), \Psi(p(x+y)),&\Psi(x+y)) = \\
( \dots, \Psi(p^{2}x),& \Psi(px),\Psi(x)) + ( \dots, \Psi(p^{2}y), \Psi(py),\Psi(y))
 \; .
\end{split}
\eeq
Equivalently, $\Psi$ satisfies the addition law \cite[(11)]{psi}
\beq
\label{covsum2} 
\Psi(x+y) =\Phi (\Psi(x), \Psi(px),\dots;\Psi(y),\Psi(py),\dots)  
\eeq
where
\beq
\label{covsum3}  \begin{split}
 \Phi (\Psi(x), \Psi(px),\dots;\Psi(y),&\Psi(py),\dots)  = \\ \lim_{i \to \infty} 
 \varphi_i(\Psi(p^ix),&\dots,\Psi(px),\Psi(x);\Psi(p^i y),\dots,\Psi(py),\Psi(y)) \;,
 \end{split}
\eeq
for the polynomials $\varphi_i$ of \eqref{limphi0} and \eqref{Phidef0}.
Notice that   \eqref{wittsum} may be restated to say that,  for any $x \in \Q_p$, 
$$x = (\dots , x_{-2},x_{-1};x_0, x_1,\dots)\;,
$$ where $x_i \in \F_p$ is given by  \eqref{wittcomp2}, as a \emph{Witt bivector} \cite{MA} with coefficients in $\F_p$. 
\par
\subsection{Our present approach}
We present here in section~\ref{proofs}   direct elementary proofs of the main properties of $\Psi$, which 
make  no use of the Barsotti-Witt algorithms of \cite{MA}. 
Actually, following a suggestion of M. Candilera, we  consider rather than \eqref{functeq}, the more general functional  relation for $\Psi = \Psi_q$, $q =p^f$,
\beq  \label{functeq(q)}
\sum_{j=0}^{\infty} p^{-j}\Psi(p^{j} T)^{q^{j}} = T \; .
\eeq
The result, at no extra work,  will then be that \eqref{functeq(q)} admits a unique solution 
$\Psi_q(T) \in T + T^2 \Z[[T]]$. The series $\Psi_q(T)$  represents a $p$-adically entire function such that 
 $\Psi_q(\Q_q) \subset \Z_q$. 
In section~\ref{newton} we describe in the same elementary style the Newton and 
valuation polygons of  the entire function $\Psi_q$, and obtain new 
estimates on the growth of $|\Psi_q(x)|$ as $|x| \to \infty$, which will be crucial for the sequel \cite{perf_fourier}. 
From these estimates we also deduce, modulo the self-contained technical Proposition~\ref{estadm} whose proof appears in \cite{perf_fourier}, the location of the zeros of $\Psi_p$ (Theorem~\ref{invfunct}). 
Namely, 
any ball of radius 1, $a + \Z_p \in \Q_p/\Z_p$, contains precisely one (simple) zero of $\Psi_p$. 

\par
We present in Appendix C below some  numerical calculations due to M. Candilera,  which exhibit the first coefficients of $\Psi_p$, for small values of $p$. These 
calculations have been  useful to us and we believe they may be quite convincing for the reader.
\par  
The function $\Psi_q: \A_{\Q_p}^1 \to \A_{\Q_p}^1$ is a quasi-finite  covering of the Berkovich affine line over $\Q_p$ by itself. We do not know 
whether the previous covering  is Galois.  
\subsection{Convergence of Fourier-type expansions}\label{reconcile}
Section~\ref{entbddfcts} describes some Hopf algebras whose existence follows from the addition properties of $\Psi_p$. The next section~\ref{almper} suggests an interpretation of the functions $\Psi_p(x/\lambda)$, for 
$\lambda \in \Q_p^\times$,  as $p$-adic analogs of $\exp(\frac{2 \pi i}{\lambda} z)$, for $\lambda \in \R^\times$. 
We are naturally lead to the question of which functions can be expressed as uniform limits on $\Q_p$ of the previous functions. By analogy to the classical case, we call these functions uniformly almost periodic  on $\Q_p$ and denote by $AP_{\Q_p}$ the corresponding closed  subalgebra of the Banach algebra $\sC_\unif^\bd(\Q_p,\Q_p)$ of bounded uniformly continuous functions 
$\Q_p \to \Q_p$. Although we do not have an intrinsic characterization of these functions, we can show that 
they may be seen as germs of holomorphic functions on a neighborhood of $\Q_p$. 
We point out that colimits for topological algebras are not in general supported by set-theoretic inductive limits (see Remark~\ref{caution} below). Therefore, our Uniform Approximation Theorem~\ref{bohrpadcor} does not state that 
any uniformly almost periodic function  on $\Q_p$ necessarily extends to an analytic function on a  $p$-adic strip around $\Q_p$. 
On the other hand, $AP_{\Q_p}$ is dense in the Fr\'echet algebra $\sC(\Q_p,\Q_p)$ of continuous  
functions  $\Q_p \to \Q_p$, equipped with the  topology of uniform convergence on compact open subsets of $\Q_p$. 
The proofs of these facts are detailed in  sections~\ref{contfcts} and \ref{entfcts}. 
We spend some time in 
section~\ref{contfcts} to explain in categorical terms (clearly stated in Appendix A) the natural limit/colimit/tensor product formulas 
which justify the linear topologies of the previous function algebras.  
For example,  $\sC(\Q_p,\Z_p)$ (but not $\sC_\unif(\Q_p,\Z_p)$) is a
 Hopf algebra  related to the constant $p$-divisible group $\Q_p/\Z_p$ over $\Z_p$ and to its ``universal covering'' $\Q_p$.  A more complete discussion of these topological algebras and of their duality relation with the affine algebra of the universal covering of the $p$-divisible torus, interpreted as an algebra of measures,  will appear in \cite{duality}.  
 \endgraf
  In section~\ref{entfcts} we prove the facts announced in section~\ref{almper},
  namely Theorem~\ref{bohrpadthm}, Theorem~\ref{bohrpadthm2},
Proposition~\ref{2struct},   Proposition~\ref{hopfpad}, Proposition~\ref{fejer21}, Proposition~\ref{fejer22}, 
  and Theorem~\ref{bohrpadcor}.
\subsection{Aknowledgments} 
It is a pleasure to acknowledge that the proofs in sections \ref{proofs} and \ref{newton} of this  text are based on a discussion  with Philippe Robba which took place in April 1980.  
I am strongly indebted to him for this and for his friendship.
\par\noindent I thank my colleague Giuseppe De Marco for his patient explanations on classical Fourier theory.
\par \noindent 
 I  thank the MPIM of Bonn for  hospitality during March 2018, when this article was completed.

\begin{section}{Rings of $p$-adic analytic functions} \label{anftcs}
\begin{subsection}{Entire functions bounded on $p$-adic strips} \label{entbddfcts} 
(See Appendix A for notation of topological algebra and non-archimedean functional analysis.) We describe here the Hopf algebra object $\Q_p\{x\}$ in the category of Fr\'echet $\Q_p$-algebras equipped with the completed projective = inductive tensor product $\wt_{\pi,\Q_p} = \wt_{\iota,\Q_p}$, which consists of the global sections of the $\Q_p$-analytic group $\G_a$.  We also consider boundedness conditions for the functions in $\Q_p\{x\}$ 
on suitable neighborhoods of $\Q_p$ in the Berkovich affine line $\A^1_{\Q_p}$ over $\Q_p$. 
\endgraf
Our notation for  coproduct (resp. counit, resp. inversion) of a Hopf algebra object $A$ in a symmetric monoidal category with 
monoidal product $\otimes$ and unit object $I$
is usually $\P =\P_A:A \to A \otimes A$ (resp. $\veps = \veps_A :A \to I$, resp. $\rho = \rho_A:A \to A$). 
\endgraf

\begin{defn} \label{entfctsdef}
For any closed subfield $K$ of $\C_p$,   we denote by $K\{x\} = K\{x_1,\dots,x_n\}$ the ring of entire functions on the $K$-analytic affine space $(\A^n_K,\cO_K)$. 
The \emph{standard Fr\'echet topology} on the $K$-algebra $K\{x\}$  is induced by the family $\{w_r\}_{r \in \Z}$  of valuations  
$$w_r (f) := \inf_{x \in (p^{-r} \C_p^\circ)^n} v(f(x)) \;,
$$
for any  $f \in K\{x\}$.   
\end{defn} 
\begin{rmk}\label{supnormnot} More generally, for bounded functions $f:X \to (S,|~|)$,  where $X$ is a set and $(S,|~|)$ is a Banach ring in multiplicative notation, $||f||_X = \sup_{x \in X} |f(x)|$ will denote the supnorm on $X$. 
\end{rmk}
\begin{defn} \label{padicstrip}
For any $\rho>0$ and any finite extension $K/\Q_p$, the \emph{$p$-adic $n$-strip of width $\rho$ around $K^n$} is the analytic domain which is the union $\Sigma_\rho(K) = \Sigma^{(n)}_\rho(K)$ of all affinoid  $n$-polydiscs of radius $\rho$ centered at $K$-rational points.  
We denote by
$$
Res_\rho : \C_p\{x\} \longrightarrow \cO(\Sigma_\rho) \;,\; f \longmapsto f_{|\Sigma_\rho} 
$$
the restriction map. 
Clearly, the map $Res_\rho$ is an injection. We let $\cO_K^\bd(\Sigma_\rho(K))$ (resp. $\cO_K^\circ(\Sigma_\rho(K))$) denote the subring of $\cO_K(\Sigma_\rho(K))$, consisting of functions bounded (resp.   bounded  by 1) on $\Sigma_\rho(K)$. 
We denote by $||~||_{K,\rho}$ the supnorm on $\Sigma_\rho(K)$. The Banach algebra structure on $\cO_K^\bd(\Sigma_\rho(K))$ (resp. $\cO_K^\circ(\Sigma_\rho(K))$) induced by the norm $||~||_{K,\rho}$ will be called \emph{$K$-uniform}. The   Fr\'echet structure of $\cO_K(\Sigma_\rho(K))$ (resp. $\cO_K^\circ(\Sigma_\rho(K))$) induced by the family of seminorms of Definition~\ref{entfctsdef} will be called \emph{standard}.
We set in  particular
$\Sigma_\rho = \Sigma^{(1)}_\rho(\Q_p)$ but will keep the notation  $||~||_{\Q_p,\rho}$.
We also denote by $\sH^{(n),\bd}_K(\rho)$ 
(resp. $\sH^{(n),\circ}_K(\rho)$) 
the subring of $K\{x\}$
of functions which are bounded 
(resp. bounded by 1) 
on $\Sigma_\rho(K)$.  
We set 
$$ \sH^{(n),\bd}_K := \bigcap_{\rho} \sH^{(n),\bd}_K(\rho) \;.
$$  
For any $\rho >0$ and any 
$f \in \sH^{(n),\bd}_K(\rho)$ we introduce one further valuation
\beq
\label{entbddfctsdef1}
w_{K,\infty} (f) := \inf_{x \in K^n} v(f(x)) \;.
\eeq
For $n=1$ and $K =\Q_p$, we shorten $\sH^{(n),\bd}_K(\rho)$ 
(resp. $\sH^{(n),\circ}_K(\rho)$, 
resp. $\sH^{(n),\bd}_K$, 
resp. $w_{K,\infty}$, resp.  $K$-uniform) to $\sH^\bd(\rho)$ 
(resp. $\sH^\circ(\rho)$, 
resp. $\sH^\bd$, 
resp. $w_\infty$, resp. uniform).
 \end{defn}  
  \begin{rmk}\label{bddexist} It is not a priori clear that $\sH^\bd$ contains non-constant functions. We will prove below (Theorem~\ref{psibddsuper}) that $\Psi(x) \in \sH^\bd$.   
   \end{rmk}
 \begin{rmk}\label{bddent} 
 For any $n$ and any $\rho >0$, $\sH^{(n),\circ}_K(\rho)$ is a closed $K^\circ$-subalgebra of 
 $K\{x_1,\dots,x_n\}$; the induced Fr\'echet $K^\circ$-algebra structure on $\sH^{(n),\circ}_K(\rho)$ will be called \emph{standard}. 
 It follows from formula~\ref{functeq} below that, by contrast, 
 $\sH^{(n),\bd}_K(\rho) = \sH^{(n),\circ}_K(\rho)[1/p]$ is dense in $K\{x_1,\dots,x_n\}$. 
 \end{rmk} 
 \begin{rmk} \label{projstrips} The Fr\'echet structure on $\cO_K(\Sigma_{\rho})$ which we call ``standard'' 
 is the one of analytic geometry: it coincides with the topology of uniform convergence on rigid discs of radius $\rho$.  
Similarly for 
$\cO_K^\circ(\Sigma_\rho(K))$.
 The standard Fr\'echet algebra $K\{x\}$ identifies with 
\beq \label{projstrips1}
K\{x\} = \limPROu_{\rho \to +\infty} (\cO_K(\Sigma_{\rho}),{\rm standard})
\eeq
\end{rmk}
\begin{defn} \label{projstrips2} 
The \emph{strip topology} on $\sH^{(n),\bd}_K$ 
 is the projective limit topology 
of the uniform topologies of Definition~\ref{padicstrip}.
So,
\beq \label{projstrips21}  
(\sH^{(n),\bd}_K, {\rm strip}) = \limPROu_{\rho \to +\infty} (\cO_K^\bd(\Sigma_\rho(K)), ||~||_{K,\rho}) \;,
\eeq
is a $K$-Fr\'echet space. 
\end{defn}
\begin{rmk}\label{stripvsstan} We have a dense embedding $\sH^{(n),\bd}_K \subset K\{x_1,\dots,x_n\}$. The strip topology on $\sH^{(n),\bd}_K$, for which this algebra is complete, is finer than its (non complete) standard topology.
\end{rmk}
The next lemma shows that, for any non archimedean field $K$ and $\G_a = \G_{a,K}$, 
$$\cO(\G_a \times \G_a) = \cO(\G_a) \wt_{\pi,K} \cO(\G_a)
$$ 
so that  
$\cO(\G_a)$ is a Hopf algebra object in the category of Fr\'echet $K$-algebras. 
\begin{lemma} \label{tensorident0} There are
 natural identifications 
 \beq \label{ident1}
K\{x_1,\dots,x_n \} \wt_{\pi,K} K\{y_1,\dots,y_m\} =  K\{  x_1,\dots,x_n,y_1,\dots,y_m \} \; .
\eeq
sending $x_i \otimes 1 \mapsto x_i$ and $1 \otimes y_j \mapsto y_j$, for $i = 1,\dots,n$, $j=1,\dots,m$.
\end{lemma}
\begin{proof}
For any $s \in \Z$ the map of the statement produces  isomorphisms of $K$-Tate algebras \cite[\S 6.1.1, Cor. 8]{BGR}
  \beq \label{ident0}
K \langle p^{-s}x_1,\dots,p^{-s}x_n \rangle  \wt_{\pi,K} K \langle p^{-s}y_1,\dots,p^{-s}y_m  \rangle =  K \langle p^{-s}x_1,\dots,p^{-s}x_n,p^{-s}y_1,\dots,p^{-s}y_m \rangle \;.
\eeq 
We now apply Proposition~\ref{tensorfrechet2}.
\end{proof} 
\begin{cor}  \label{tensoridentbdd} Let $K$ be a finite extension of $\Q_p$ and $\rho>0$. The identifications \eqref{ident1} induce identifications 
 \beq \label{ident2}
( \sH^{(n),\circ}_K(\rho), {\rm standard})  \wt^u_{K^\circ}  (\sH^{(m ),\circ}_K(\rho), {\rm standard}) \iso  (\sH^{(m+n),\circ}_K(\rho), {\rm standard}) \;.
\eeq
Similarly for $(\cO_K^\circ(\Sigma_\rho(K)), {\rm standard})$. 
\end{cor}
\begin{cor}\label{hopfbdd}
The map $\P:x_i \mapsto x_i \wt 1 + 1 \wt x_i$ makes $K\{x_1,\dots,x_n \}$ into a Hopf algebra object 
in the category of Fr\'echet $K$-algebras. The restriction of $\P$ to $\sH^{(n),\circ}_K(\rho)$ induces a map
$$
\P: (\sH^{(n),\circ}_K(\rho), {\rm standard}) \longrightarrow (\sH^{(n),\circ}_K(\rho), {\rm standard}) \wt_{K^\circ}^u (\sH^{(n),\circ}_K(\rho), {\rm standard})
$$
which makes $(\sH^{(n),\circ}_K(\rho), {\rm standard})$ a Hopf algebra object 
in the category of Fr\'echet $K^\circ$-algebras. Similarly for $(\cO_K^\circ(\Sigma_\rho(K)), {\rm standard})$. 
\end{cor}
\end{subsection}
\subsection{$p$-adic almost periodic functions} \label{almper}
We sketch  here the   the main ideas and results on $p$-adic almost periodic functions. Proofs are given in section~\ref{entfcts} below. We freely use in this introduction the (quite self-explanatory) notation of section~\ref{contfcts} for continuous, uniformly continuous, bounded rings of $p$-adic functions $\Q_p \to \Q_p$ and their topologies. 
\par
The following elementary lemma shows that a naive $p$-adic analog of real Bohr's uniformly almost periodic functions (see Definition~\ref{bohrdef} in Appendix B), where ``an interval of length $\ell_\veps$ in $\R$'' is taken to mean a coset $a + p^h\Z_p$, for $a \in \Q_p$ and $p^{-h} = \ell_\veps$, does not lead to a meaningful definition.
\begin{lemma}
\label{naivedef} A continuous function $f: \Q_p \to \Q_p$  which has the property that for any $\veps>0$, there exists $h \in \Z$ such that any coset $a + p^h \Z_p$ in $\Q_p/p^h \Z_p$ contains an element $t_a$ such that
\beq \label{naivedef1} 
|f(x+t_a)-f(x)| < \veps \;\;\forall \; x \in \Q_p\;,
\eeq
is constant. 
\end{lemma}
\begin{proof}  In fact, from condition \eqref{naivedef1},  for any  $a \in \Q_p$, it follows by iteration that $t_a$ may be replaced by any $t \in \Z t_a$. By continuity, we may replace $t_a$ by any $t \in  \Z_p t_a$. For $a \notin p^h \Z_p$, $\Z_p t_a = \Z_p a$. So, if we pick $a = p^{-N}$, for   $N >>0$,  \eqref{naivedef1} 
implies that the variation of $f(x)$ in $p^{-N} \Z_p$ is less than $\veps$. So, the variation of $f(x)$ in $\Q_p$ is less than $\veps$ for any $\veps>0$, hence $f$ is constant. 
\end{proof}
 \endgraf   We resort to an \emph{ad hoc} definition. For $x \in \Q_p$, let us consider the classical Witt vector expression 
\beq \label{wittsum}
x = \sum_{i >> -\infty}^\infty [x_i]\,p^i \in {\rm W}(\F_p)[1/p]  =\Q_p\; ,
\eeq
where $[t]$, for $t \in \F_p$, is the Teichm\"uller representative of $t$ in ${\rm W}(\F_p) = \Z_p$. Notice that, for any $i \in \Z$,
the function
\beq \label{wittcomp2}
x_i : \Q_p \longrightarrow \F_p \;\;,\;\; x \longmapsto x_i
\eeq
factors through a function, still denoted by $x_i$,
\beq \label{wittcomp3}
x_i :  \Q_p/ p^{i+1}\Z_p \longrightarrow \F_p \;\;,\;\; h \longmapsto h_i \;.
\eeq
We regard the function in \eqref{wittcomp3}  as an $\F_p$-valued  \emph{periodic function  of period $p^{i+1}$} on $\Q_p$. In the following, for any $i \in \Z$ and any $\lambda \in \Q_p^\times$, we denote by ``$[(\lambda x)_i]$''  the uniformly continuous function $\Q_p \to \Z_p$, $x \mapsto [(\lambda x)_i]$. We observe that 
$$
[(\lambda p^jx)_i] = [(\lambda x)_{i-j}]
$$
for any $i,j \in \Z$ and $\lambda \in \Q_p^\times$.
\begin{defn} \label{seriousdef} We define the $\Q_p$-algebra  $AP_{\Q_p}$ (resp. the $\Z_p$-algebra  $AP_{\Z_p}$) of (resp. \emph{integral}) \emph{uniformly almost periodic} (\emph{u.a.p.} for short) functions $\Q_p \to \Q_p$ (resp. $\Q_p \to \Z_p$) as the closure of 
$$\Q_p[[(\lambda x)_i]\,|\, i \in \Z\,,\, \lambda \in \Z_p^\times]\;\;\mbox{(resp. of} \; \; \Z_p[[(\lambda x)_i]\,|\, i \in \Z\,,\, \lambda  \in \Z_p^\times] \;\mbox{)}
$$
in the $\Q_p$-Banach algebra $\sC^\bd_\unif(\Q_p,\Q_p)$ (resp. in the $\Z_p$-Banach ring $\sC_\unif(\Q_p,\Z_p)$), equipped with the induced valuation $w_\infty$.    
\end{defn}
\begin{rmk}\label{p-adicfract} This remark is made to partially justify Definition~\ref{seriousdef}.  
For any $N \in \Z$ we denote by $S_N : \Q_p \to p^N \Z_p$ the function 
\emph{$N$-th order fractional part}, namely 
\beq \label{wittsum12}
x = \sum_{i >> -\infty}^\infty [x_i]\,p^i \longmapsto S_N(x) = \sum_{i = N}^\infty [x_i]\,p^i \;.
\eeq
It is clear that, for any $N$ and $\lambda\in \Q_p^\times$, $x \mapsto S_N(\lambda x)$ is a bounded uniformly continuous  function. The function $S_3$, certainly not periodic,  is a $p$-adic analog of the function
$$\R \to [0,1)\;\;,\;\; ....1234.56789.... \mapsto 0.789... 
$$
which is genuinely periodic of period $0.01$. 
\end{rmk}
\endgraf
We will prove the following partial analog to  Bohr's ``Approximation Theorem'' (Theorem~\ref{bohrthm} in Appendix B), where in fact the functions $\cos(\frac{2 \pi}{\lambda} x)$ and $\sin(\frac{2 \pi}{\lambda} x)$, for $\lambda \in \R^\times$ are replaced by the functions $\Psi(\lambda x)$, for  $\lambda \in \Q_p^\times$.
\begin{thm} \label{bohrpadthm} 
$(AP_{\Q_p},w_\infty)$ (resp.   $(AP_{\Z_p}, w_\infty)$) is the completion of the valued ring 
$$(\Q_p[\Psi(\lambda x)\,|\, \lambda \in \Q_p^\times],w_\infty)\;\;\mbox{(resp.} \; 
(\Z_p[\Psi(\lambda x)\,|\, \lambda \in \Q_p^\times],w_\infty)\,\mbox{)}\;.
$$
\end{thm}
  \begin{defn}\label{Elambdadef}
  For any $\lambda \in \Q_p^\times$, the Fr\'echet $\Z_p$-algebra $\sE^\circ_\lambda$ (resp. $\sT^\circ_\lambda$)
  is the closure of 
\beq \label{Elambdadef2}
\Z_p[\Psi(\lambda^{-1}p^j x)\,|\,j=0,1,\dots]
\eeq
 in $\Q_p\{x\}$ (resp. in $\cO(\Sigma_{|\lambda|})$) with the standard topology.
 We then set  $\sE^\bd_\lambda := \sE^\circ_\lambda[1/p]$ 
 (resp. $\sT^\bd_\lambda := \sT^\circ_\lambda[1/p]$).
\par Finally, we define the Fr\'echet $\Z_p$-algebra $\sE^\circ$ as the closure of $\Z_p[\Psi(\lambda^{-1}p^j x)\,|\,j=0,1,\dots]$ in 
$\Q_p\{x\}$, and set $\sE^\bd := \sE^\circ[1/p]$.   \end{defn}

  \begin{thm} \label{bohrpadthm2} {\bf (Approximation Theorem on compacts)}  The completion of the multivalued ring 
$$(\sE^\bd, \{||~||_{p^r\Z_p}\}_{r \in \Z}) \;\;\mbox{(resp.} \; 
(\sE^\circ, \{||~||_{p^r\Z_p}\}_{r \in \Z})\,\mbox{)} 
$$
is the Fr\'echet $\Q_p$-algebra (resp. $\Z_p$-algebra) $\sC(\Q_p,\Q_p)$ (resp. $\sC(\Q_p,\Z_p)$). 
\end{thm}
  The following proposition follows from the estimates of Proposition~\ref{unifPsi} (see Corollary~\ref{basicest0} or Theorem~\ref{psibddsuper} for the  proof) together with the fact that the conditions listed below are closed for the standard Fr\'echet structure.
  The proof of the latter fact is given  in Lemma~\ref{standardvsunif}. 
 \begin{prop}\label{2struct}  For any $f \in \sE^\circ_\lambda$
 (resp. $f \in \sT^\circ_\lambda$)
 we have
 \ben
  \item $f$ is bounded by 1 on the $p$-adic strip $\Sigma_{|\lambda|}$;
\item $f(\Q_p) \subset \Z_p$;
  \item For any $r \in \Z$, $a,j \in \Z_{\geq 0}$, the function $g(x) := f(p^{-r}x)^{p^a}$ 
  satisfies 
  $$g(x + p^{r+j} \lambda \C_p^\circ) \subset g(x) + p^{a+j} \C_p^\circ \;\;,\;\; \forall \; x \in \Q_p\;.
  $$
  \een  
  \end{prop}

 \begin{prop} \label{hopfpad} For any $\lambda \in \Q_p^\times$,   $(\sE^\circ_\lambda, {\rm standard})$   (resp. 
 $(\sT^\circ_\lambda, {\rm standard})$)
is a
 Hopf algebra object  in the monoidal category  
$(\LMu_{\Z_p},\wt^u_{\Z_p})$ for the coproduct $\P$ and coidentity $\veps$ given by 
\beq  \label{hopfpad2}
\P(\Psi (\lambda^{-1} p^j x)) \mapsto  \Psi (\lambda^{-1} p^j x \wt 1 + 1 \wt \lambda^{-1} p^j x) \;\;,\;\; \veps (\Psi (\lambda^{-1} p^j x)) = 0\;, 
\eeq
 for $j=0,1,\dots$.
 This structure only depends  upon $|\lambda|$. 
\end{prop}
\begin{defn} \label{reptilambda}
We define  $\E_\lambda$ in \eqref{tilambda} (resp. $\T_\lambda$)
as  the abelian group functor on $\ACLMu_{\Z_p}$, represented by the Hopf algebras
 $(\sE^\circ_\lambda, {\rm standard})$ (resp. by $(\sT^\circ_\lambda, {\rm standard})$). 
  \end{defn}

 A partial $p$-adic analog of F\'ejer's Theorem, or, more precisely, of Theorem~\ref{fejer2} in Appendix B,  is then
 \begin{prop} \label{fejer21}
 For any $\lambda \in \Q_p^\times$, the completion of the valued ring 
$$(\Z_p[\Psi(\lambda^{-1}p^j x)\,|\,j=0,1,\dots],w_\infty) 
$$ coincides with its closure in $\sC_\unif(\Q_p,\Z_p)={\rm W}(\sC_\unif(\Q_p,\F_p))$ equipped with the $p$-adic topology,  and identifies with ${\rm W}(\F_p[[(\lambda^{-1}x)_{-j}]\,|\, j=0,1,\dots])$ also equipped with the $p$-adic topology. 
 \end{prop}
  For the standard topology we have
 \begin{prop} \label{fejer22}
  For any $\lambda \in \Q_p^\times$, the completion of the valued ring  
   $(\sE_\lambda^\circ,w_\infty)$ (resp. $(\sT_\lambda^\circ,w_\infty)$) 
  coincides with its closure in $\sC(\Q_p,\Z_p)={\rm W}(\sC(\Q_p,\F_p))$ equipped with the product topology of the prodiscrete  topologies on the components \eqref{tesiform21},  and identifies with 
  ${\rm W}(\F_p(v(\lambda),\infty))$ (see Proposition~\ref{tesiunif} below for notation) also equipped with the product topology of the prodiscrete  topologies on the components. 
  \par
The
 ring $\sE_\lambda^\circ \otimes_{\Z_p} \F_p$ (resp. $\sT_\lambda^\circ \otimes_{\Z_p} \F_p$) equipped with the quotient topology coincides with 
 $\F_p(v(\lambda),\infty) = \sC(\Q_p/\lambda p \Z_p,\F_p)$. 
 \end{prop}
 
 We now introduce our  $p$-adic   analog  of the sheaf  $\APH_{\C}$ 
 of  almost periodic  analytic functions (see subsection~\ref{expvstrig} in Appendix B).
%
 \begin{defn}\label{padAPHdefamel} 
 \hfill 
 \ben
 \item For any $\rho >0$, we define the  algebra of (resp. \emph{integral}) \emph{almost periodic $p$-adic analytic functions on the strip $\Sigma_{\rho}$} as 
the closure $\APH_{\Q_p}(\Sigma_{\rho})$ (resp. $\APH_{\Z_p}(\Sigma_{\rho})$) of 
$\Q_p[\Psi(\lambda x) \,|\, \lambda \in \Q_p^\times]$  (resp. $\Z_p[\Psi(\lambda x) \,|\, \lambda \in \Q_p^\times]$) in 
$(\cO^\bd(\Sigma_{\rho}),{\rm uniform})$, with the induced Banach ring structure. 
 \item  The algebra of  \emph{germs at $0$ of  almost periodic $p$-adic analytic functions} is the locally convex inductive limit
 \beq
( \APH_{0,\Q_p},{\rm strip}) := \limINDu_{\rho \to 0} \APH_{\Q_p}(\Sigma_{\rho}) \;.
 \eeq 
   \item  The algebra of  \emph{germs at $0$ of integral almost periodic $p$-adic analytic functions} is 
 \beq
( \APH_{0,\Z_p},{\rm strip}) := \limINDU_{\rho \to 0} \APH_{\Z_p}(\Sigma_{\rho}) \;.
 \eeq 
  \item 
  The algebra of   \emph{almost periodic $p$-adic  entire  functions} is  
   \beq
 (APH_{\Q_p} ,{\rm strip}):= \limPROu_{\rho \to +\infty} \APH_{\Q_p}(\Sigma_{\rho}) \;.
  \eeq
  \item
  The algebra of \emph{integral  almost periodic $p$-adic  entire  functions} is the closure $(APH_{\Z_p} ,{\rm strip})$ of $\Z_p[\Psi(\lambda x) \,|\, \lambda \in \Q_p^\times]$ in $(APH_{\Q_p} ,{\rm strip})$ equipped with the induced Fr\'echet $\Z_p$-algebra 
  structure.  
  \item The Fr\'echet $\Z_p$-algebra $\sE^\circ$ is a Hopf algebra object in the category $\LMu_{\Z_p}$ for the laws \eqref{hopfpad2}.
The corresponding group functor 
 \beq \label{univcov2}
 \E : \ACLM^u_{\Z_p} \longrightarrow \Ab
 \eeq
 will be called
 the \emph{universal covering} of $\E_\lambda$, for any $\lambda \in \Q_p^\times$. 
   \een
  \end{defn}
   \begin{rmk}\label{entireQp} The special fiber of $\E$ is the constant group 
   $$\Q_p = \limPROu_{|\lambda| \to 0} \Q_p/\lambda \Z_p
   $$ 
   over $\F_p$. On the other hand, equation~\ref{functeq} shows that $\sE^\circ[1/p]$ is dense in $\Q_p\{x\}$, so that the  generic fiber of $\E$ is $\G_{a,\Q_p}$.
 \end{rmk}
Our Definition~\ref{padAPHdefamel} is designed as to make   the analog of 
 Theorem~\ref{bochnerthm} in Appendix B a true statement. 
In the $p$-adic case, we actually get the following more precise statement.
\begin{thm}  \label{bohrpadcor}
{\bf (Uniform Approximation Theorem)} The natural $\LMu_{\Q_p}$-morphism (resp. 
$\LMu_{\Z_p}$-morphism)
$$
( \APH_{0,\Q_p},{\rm strip})  \longrightarrow (AP_{\Q_p}, w_\infty)
$$
(resp. 
$$
 (\APH_{0,\Z_p},{\rm strip})   \longrightarrow (AP_{\Z_p}, w_\infty) \;\mbox{)}\;,
$$
is an isomorphism. 
\end{thm} 
  \endgraf
 The similarity with classical Fourier expansions will be made more stringent in \cite{perf_fourier}, where  the classical Mahler 
binomial expansions of continuous functions $\Z_p \to \Z_p$ is generalized to an expansion of any uniformly continuous functions $\Q_p \to \Q_p$ as a series with countably many terms of entire functions of exponential type. Such a $p$-adic Fourier theory on $\Q_p$  presents the same power and limitations as the classical Fourier theory on $\R$. 
Functions in $AP_{\Q_p}$  play the role of Bohr's uniformly almost periodic functions and a variation of the Bochner-Fej\'er approximation theorem \cite[I.9]{Bes} holds. 
On the other hand, a Fourier series $\sF(f)$ (with countably many terms) does exist for a much more general class of functions $f:\Q_p \to \Q_p$ and the classical question as to what extent the series $\sF(f)$ approximates $f$ makes perfect sense, precisely as in classical Harmonic Analysis.   
\endgraf
 We ask whether 
the classical Bohr compactification of $\Q_p$ has a $p$-adic analytic description, as it has one in terms of classical (\ie complex-valued) harmonic theory on the locally compact group $(\Q_p,+)$. 
\par
We expect that a completely analogous theory should exist for any finite extension $K/\Q_p$. To develop it 
properly it will be necessary to extend  Barsotti covector's construction to ramified Witt vectors modeled on $K$
and to relate this construction to  Lubin-Tate groups over $K^\circ$ \cite{schneider2}. 

\section{Elementary proofs of the main properties of $\Psi$} \label{proofs}
We prove here the basic properties of the function $\Psi$. In contrast to \cite{psi}, the proofs are here completely self-contained.
\begin{prop} \label{functeq2}
The equation \eqref{functeq(q)} has a unique solution in $\Psi = \Psi_q \in T+ T^2\Z[[T]]$\;.
\end{prop}
\begin{proof}
We endow $\Z[[T]]$ of the $T$-adic topology. It is clear that, for any $\varphi \in T\Z[[T]]$,
the series $\sum_{j=1}^{\infty} p^{-j}\varphi(p^{j} T)^{q^{j}}$ converges in $T\Z[[T]]$. 
Moreover, the map
\beq \label{contr}
\sL: \varphi \longmapsto T - \sum_{j=1}^{\infty} p^{-j}\varphi(p^{j} T)^{q^{j}} \;,
\eeq
is a contraction of the complete metric space $T+T^2\Z[[T]]$. In fact, let $\veps(T) \in T^r\Z[[T]]$, with $r \geq 3$. For any $\varphi \in T+T^2\Z[[T]]$ we see that 
$$
\sL(\varphi + \veps) - \sL(\varphi) \in T^{r(q-1)+q}\Z[[T]]\;.
$$
 Since $r(q-1)+q > r$ this shows that $\sL$ is a contraction. 
So, this map has a unique fixed point which is $\Psi_q(T)$.
\end{proof}
The following proposition, due to M. Candilera, provides an alternative proof of Proposition~\ref{functeq2} and finer information on $\Psi_q(T)$.
\begin{prop} \label{p-1} {\rm (M. Candilera)} The functional equation for the unknow function $u$
\beq \label{functeq(uq)}
1 = \sum_{j=0}^\infty p^{j (q^j-1)} T^{\frac{q^j-1}{q-1}}u(p^{j(q-1)}T)^{q^j}  
\eeq
admits a unique solution  $u(T) = u_q(T) \in 1 + T\Z[[T]]$. We have
\beq \label{formulaCand} \Psi_q(T) = Tu_q(T^{q-1})\;.
\eeq
\end{prop}
\begin{proof} In this case we consider the $T$-adic metric space $1 + T\Z[[T]]$ and the map
\beq \label{contr2} \begin{split}
\sM:   1 + T\Z[[T]]  &\longrightarrow 1 + T\Z[[T]] 
\\
\varphi & \longmapsto 1 - \sum_{j=1}^\infty p^{j (q^j-1)} T^{\frac{q^j-1}{q-1}}\varphi(p^{j(q-1)}T)^{q^j}  \;.
\end{split}
\eeq
We endow $\Z[[T]]$ of the $T$-adic topology. It is clear that, for any $\varphi \in T\Z[[T]]$,
the series $\sum_{j=1}^{\infty} p^{-j}\varphi(p^{j} T)^{q^{j}}$ converges in $T\Z[[T]]$. 
If  $\veps(T) \in T^r\Z[[T]]$, with $r \geq 2$. For any $\varphi \in 1+T\Z[[T]]$ we see that 
$$
\sM(\varphi + \veps) - \sM(\varphi) \in T^{r+1}\Z[[T]]\;.
$$ 
So, the map $\sM$ is a contraction and its unique fixed point has the properties stated for the series $u$ in the statement.
\end{proof}

\begin{prop} \label{entire}
The series $\Psi(T) = \Psi_q(T)$ is entire.
\end{prop}
\begin{proof}
Since $\Psi \in T + T^2\Z[[T]] \subset T\Z[[T]]$, we deduce that $\Psi$ converges for $v_p(T)>0$. Since the coefficient of $T$ in $\Psi(T)$ is 1, whenever $v_p(T)>0$ we have $v_p(\Psi(T)) = v_p(T)$.
\par
Suppose $\Psi$ converges for $v_p(T)>\rho$, for $\rho \leq 0$. Then, for $j \geq 1$, $\Psi(p^{j} T)^{q^{j}}$ converges for $v_p(T)>\rho-1$. Moreover,
if $j > -\rho +1$ and $v_p(T)>\rho-1$, we have
\beq \label{est1}
v_p( p^{-j}\Psi(p^{j} T)^{q^{j}}) = -j + q^j (v_p(p^{j} T))  > -j + q^j(j +\rho-1) \;,
\eeq
and this last term $\to + \infty$, as $j \to +\infty$.
\par
This shows that the series $T - \sum_{j=1}^{\infty} p^{-j}\Psi(p^{j} T)^{q^{j}}$ converges uniformly for $v_p(T) > \rho -1$, so that its sum, which is $\Psi$,  is analytic for $v_p(T) > \rho -1$. It follows immediately from this that $\Psi$ is an entire function.
\end{proof}
\begin{rmk}\label{entire2} We have  proven that, for any $j =0,1,\dots$ and for $v_p(T) > -j$,
\beq \label{est2}
v_p( p^{-j}\Psi(p^{j} T)^{q^{j}}) = -j + q^j(j + v_p(T)) \;.
\eeq
In particular, for any $a \in \Z_q$,  ($\Psi(a) \in \Z_q$ and) $\Psi_q(a) \equiv a$, modulo $p\Z_q$.
\end{rmk}
\begin{prop}\label{wittdecomp}
For any $a \in \Q_q$, $\Psi_q(a) \in \Z_q$.
\end{prop}
\begin{proof}
Let $ a \in \Z_q$. We define by induction the sequence $\{a_i \}_{i=0,1,\dots}$~:
\beq \label{cong0}
a_0 =a\;\; ,\;\; a_i = \sum_{j=0}^{i-1} p^{j-i}(a_j^{q^{i-j-1}} - a_j^{q^{i-j}}) \;.
\eeq
Since, for  any $a,b \in \Z_q$, if $a \equiv b \mod p$, then $a^{q^n} \equiv b^{q^n}  \mod p q^n$, hence modulo $p^{n+1}$, while $a \equiv a^q \mod p$,
we see that $a_i \in \Z_q$, for any $i$.
\par
We then see by induction that, for any $i$,
\beq \label{cong1}
a_i = p^{-i}(a - \sum_{j=0}^{i-1} p^j a_j^{q^{i-j}}) \;\;\mbox{or, equivalently,}\;\; a = \sum_{j=0}^i p^j a_j^{q^{i-j}} \;.
\eeq
Explicitly, if we substitute in the formula which defines $a_i$, namely
$$
p^ia_i = \sum_{j=0}^{i-1} p^ja_j^{q^{i-j-1}} - \sum_{j=0}^{i-1} p^j a_j^{q^{i-j}}
$$
the $(i-1)$-st step of the induction, namely, $\displaystyle a = \sum_{j=0}^{i-1} p^j a_j^{q^{i-j-1}} $,
we get
$$
p^ia_i = a - \sum_{j=0}^{i-1} p^j a_j^{q^{i-j}} \;,
$$
which is precisely the $i$-th inductive step.

From the functional equation \eqref{functeq(q)} and from Remark~\ref{entire2} we have, for $a \in \Z_q$ and $i=0,1,2,\dots$,
\beq \label{cong2}
\Psi(p^{-i}a) \equiv p^{-i}a -  \sum_{\ell=1}^{i} p^{-\ell} \Psi( p^\ell p^{-i}a)^{q^\ell} =  p^{-i}(a -  \sum_{j=0}^{i-1} p^j \Psi(p^{-j}a)^{q^{i-j}})
 \mod p \Z_q \;.
\eeq  
Notice that $\Psi(a) \in \Z_q$ and that, modulo $p\Z_q$, $\Psi_q(a) \equiv a = a_0$, defined as in \eqref{cong0}.  We now show 
by induction on $i$ that for $a_1,\dots, a_i,\dots$ defined as in \eqref{cong0}, 
\beq \label{cong3}
\Psi(p^{-i}a) \equiv a_i \mod p\Z_q \; ,
\eeq
which proves the statement.   
In fact,
assume $\Psi(p^{-j}a) \equiv a_j \mod p\Z_q$, for $j =0,1,\dots,i-1$, and plug this information in \eqref{cong2}. We get
\beq
\Psi(p^{-i}a) \equiv p^{-i}a -  \sum_{\ell=1}^{i} p^{-\ell} a_{i-\ell}^{q^\ell} = p^{-i}(a -  \sum_{j=0}^{i-1} p^j a_j^{q^{i-j}}) = a_i\;\;\mod p\Z_q 
 \;,
 \eeq
 which is the $i$-th inductive step.
\end{proof}
\begin{rmk}\label{bivexpa} Notice that from  \eqref{cong2} it follows that, for any $a \in p^{-n}\Z_q$,
$$
a \equiv \sum_{\ell = 0}^n p^{-\ell}\Psi_q(p^\ell a)^{q^\ell} \;\mod p\Z_q \;.
$$
The formula can be more precise using the functional equation \eqref{functeq(q)} and Remark~\ref{entire2}. 
We get, for any $a \in \Q_q$,
\beq \label{bivexpa2} 
a \equiv \sum_{\ell = 0}^{-v_p(a)+i} p^{-\ell}\Psi_q(p^\ell a)^{q^\ell} \;\mod p^{i+1}\Z_q \;, \;\forall \; i \in \Z_{\geq 0}\;.
\eeq
that is  
\beq \label{bivexpa21} 
a \equiv \sum_{\ell = 0}^i p^{-\ell}\Psi_q(p^\ell a)^{q^\ell} \;\mod p^{i +v_p(a)+1}\Z_q \;, \;\forall \; i \in \Z_{\geq -v_p(a)}\;.
\eeq
\end{rmk}
We generalize \eqref{wittsum} as 
\begin{cor}\label{wittrepr} For any $a \in \Q_q$,  let 
$$
a_i := \Psi_q(p^{-i}a) \; \mod p \Z_q \;\in \F_q \;.
$$
We have 
\beq \label{wittsum(q)}
a = \sum_{i >> -\infty}^\infty [a_i]\,p^i \in {\rm W}(\F_q)[1/p]  =\Q_q\; .
\eeq
\begin{proof} Assume first that $a \in \Z_q$. In this case  \eqref{bivexpa21}  implies
\beq \label{bivexpa211} 
a \equiv \sum_{\ell = 0}^i p^{-\ell}\Psi_q(p^\ell a)^{q^\ell} \;\mod p^{i +1}\Z_q \;, \;\forall \; i \in \Z_{\geq 0}\;.
\eeq
So, the statement follows from the following
\begin{lemma} \label{onWitt} Let $i\mapsto b_i$ and $i\mapsto c_i$, for $i=0,1,\dots$, be two sequences in $\Z_q$ such that 
$$
\sum_{j=0}^i p^j b_j^{q^{i-j}} \equiv \sum_{j=0}^i p^j c_j^{q^{i-j}}  \;\mod p^{i +1}\Z_q \;, \;\forall \; i \in \Z_{\geq 0}\;.
$$
Then $$b_i \equiv c_i \mod p\Z_q\;, \;\forall \; i \in \Z_{\geq 0}\;.
$$
\end{lemma}
\begin{proof}
Immediate by induction on $i$. 
\end{proof}
In the general case, assume $a \in p^{-n} \Z_q$. Then 
\beq \label{wittsum(q1)}
p^n a = \sum_{i =0}^\infty [\Psi_q(p^{n-i}a) \,{\rm mod}\, p\Z_q ] \,p^i \in {\rm W}(\F_q)\; .
\eeq
hence
\beq \label{wittsum(q2)}
a = \sum_{i =0}^\infty [\Psi_q(p^{n-i}a) \,{\rm mod}\, p\Z_q ] \,p^{i-n} \in p^{-n}{\rm W}(\F_q)\; .
\eeq
\end{proof}
\end{cor}
From the previous corollary, it follows that $a \in \Q_q$ has the following expression as a Witt bivector with coefficients in $\F_q$
\beq \label{wittsum(q3)}
a=  (\dots,a_{-i}^{(q/p)^i}, \dots,a_{-2}^{(q/p)^2},a_{-1}^{q/p}; a_0,a_1^p, a_1^{p^2},\dots)
\;.
\eeq
which obviously equals $(\dots,a_{-i}, \dots,a_{-2},a_{-1}; a_0,a_1, a_1,\dots)$, if $q=p$.
\begin{rmk} \label{genq} We have tried to provide a simple addition formula for $\Psi_q$ of the form \eqref{covsum2}, in terms of the same power-series $\Phi$. We could not get one, nor were we  
able to establish the relation between $\Psi_q$ and $\Psi_p$, for $q =p^f$ and $f>1$. On the other hand it is clear that Barsotti's construction of 
Witt bivectors, based on classicals Witt vectors,  extends to the $L$-Witt vectors of \cite[Chap. 1]{schneider2}, where $L/\Q_p$ denotes any fixed finite extension. In our case, we would only need  the construction of \lc in the case of the field $L =\Q_q$. We  believe that the  inductive limit of $\Z_q$-groups 
${\rm W}_{\Q_q,n} \to {\rm W}_{\Q_q,n+1}$ under Verschiebung
 $$
V: (x_{-n},\dots, x_{-1},x_0) \to (0,x_{-n},\dots, x_{-1},x_0)  
 $$
is a $\Z_q$-formal groups whose  addition law is expressed by a power-series  
$\Phi_q$ analog to Barsotti's $\Phi$. We believe that equation \eqref{covsum2} still holds true for $\Psi_q$ if we replace $\Phi$ by 
$\Phi_q$. We also believe that a generalized $\Psi$ exists for any finite extension $L/\Q_p$, with analogous properties. 
\end{rmk}
\end{section}
\section{Valuation and Newton polygons of $\Psi_q$} \label{newton}
This section is dedicated to establishing the growth behavior of $|\Psi_q(x)|$ as $|x| \to \infty$. These results 
will be essential to get the delicate estimates of \cite{perf_fourier}.  \par \smallskip
\begin{subsection}{Valuation polygon of $\Psi_q$}
We recall from \cite{Laz} that the valuation polygon of a Laurent series $f = \sum_{i \in \Z} a_i T^i$ with coefficients $a_i \in \C_p$, converging in an annulus $A:=\alpha \leq v_p(T) \leq \beta$, is the graph ${\rm Val}(f)$
of the function $\mu \mapsto v(f,\mu) := \inf_{i} (v_p(a_i) + i \mu)$, which is in fact finite along the segment $\alpha \leq \mu  \leq \beta$. The function $\mu \mapsto v(f,\mu)$ is continuous, piecewise affine, and concave on $[\alpha,\beta]$.  For any $\mu \in [\alpha,\beta]$, we have
  $v(f,\mu) = \inf \{v_p(\Psi(x))\,|\,v_p(x)=\mu\,\}$. In the case of $\Psi$, $A = \C_p$ and the segment $ [\alpha,\beta]$ is the entire $\mu$-line. For the convenience of the reader we have recalled below the relation between the valuation polygon and the Newton polygon of $f$.
  \par
  \smallskip
We prove
\begin{thm} \label{valpolpsi} The valuation polygon of $\Psi_q$ goes through the origin,  has slope 1 for $\mu >-1$, and slope $q^j$, for $-j-1 <\mu <-j$, $j=1,2,\dots$ (see Figure 1).
\end{thm}
\begin{figure}[ht] \label{figure:valpolygon}
\begin{picture}(300,200)(-50,0)     
\put(0,140){\vector(1,0){170}}       
\put(100,10){\vector(0,1){180}}         
\thicklines
\put(83,123){\line(1,1){60}}
\put(63,83){\line(1,2){20}}
\put(45,10){\line(1,4){18}}
\put(80,138){\makebox{$\bullet$}}
\put(75,148){\makebox{$-1$}}
\put(60,138){\makebox{$\bullet$}}
\put(55,148){\makebox{$-2$}}
\put(40,138){\makebox{$\bullet$}}
\put(35,148){\makebox{$-3$}}
\put(20,138){\makebox{$\bullet$}}
 \put(180,145){\makebox(0,0)[b]{$\mu$-line}}
 \put(133,150){\makebox(0,0)[b]{slope $1$}}
\put(35,120){\makebox(0,0)[b]{vertex $V_1$ at  $(-1, -1)$}}
\put(5,80){\makebox(0,0)[b]{vertex $V_2$ at  $(-2, -q-1)$}}
\put(10,0){\makebox(0,0)[b]{vertex $V_3$ at  $(-3, -q^2-q-1)$}}
\put(93,100){\makebox(0,0)[b]{slope $q$}}
\put(75,40){\makebox(0,0)[b]{slope $q^2$}}
\end{picture}
\caption{The valuation polygon of $\Psi_q$.}
\end{figure}
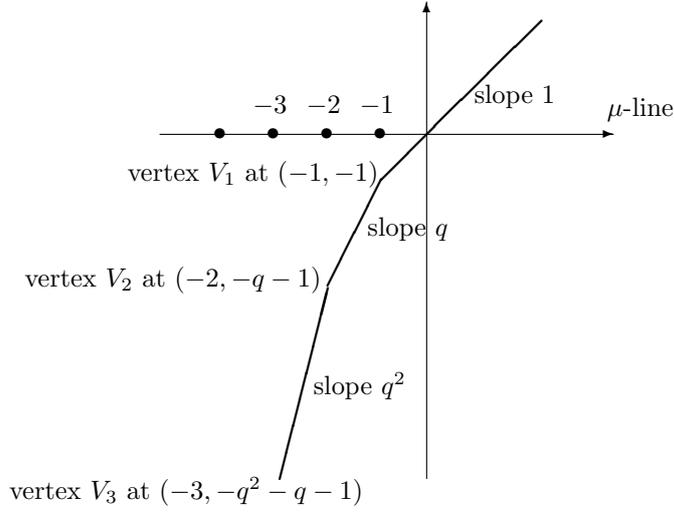
\begin{proof}
We recall that  if both $f$ and $g$ converge in the annulus $A:=\alpha \leq v_p(T) \leq \beta$, then, for any $\mu \in [\alpha,\beta]$, $v(f+g,\mu) \geq \inf(v(f,\mu),v(g,\mu))$, and that equality holds at $\mu$ if  $v(f,\mu) \neq v(g,\mu)$. Moreover, for any $n \in \N$, $v(f^n,\mu) = n\,v(f,\mu)$.
\par
In the polygon in Figure 1, for $j=1,2,\dots$, the side of projection $[-j,1-j]$ on the $\mu$-axis is the graph of the function
\beq
\sigma_j (\mu) := q^{j-1} (\mu + j-1) - q^{j-2}-\dots-q -1   \;.
\eeq
 Notice that
\beq
\sigma_{j+1}(\mu) = -1 + q\, \sigma_{j}(\mu + 1) \;,
\eeq
and therefore
\beq
\sigma_{j+i}(\mu) = -1 -q - \dots - q^{i-1} + q^i\, \sigma_{j}(\mu + i) \;,
\eeq
for any $i =0,1,2,\dots$.
\par
Since $\Psi \in T \Z[[T]]$ and since the coefficient of $T$ is 1, we have $v(\Psi,\mu) = \mu$, for $\mu \geq 0$. For $0 > \mu > -j$, $j \geq 1$, we have
\beq \label{est3} \begin{split}
v(p^{-j} \Psi(p^j T)^{q^j}, \mu) = -j + v(\Psi(p^j T)^{q^j}, \mu) = -j + q^j v(\Psi(p^j T), \mu) = \\ -j + q^j v(\Psi(S), j+\mu) =
-j + (j + \mu)q^j > \mu =  v(T,\mu)\;,
\end{split}
\eeq
where we have used the variable $S =  p^j T$.
 \begin{rmk} \label{equality}
 For $\mu = -j$ we get equality in the previous formula. \end{rmk}
 Let us set, for $j=0,1,2,\dots$,
$$\ell_j(\mu) = -j + (j + \mu)q^j \;,
$$
so that \eqref{est3} becomes
\beq \label{est31}
v(p^{-j} \Psi(p^j T)^{q^j}, \mu) = \ell_j(\mu)  > \ell_0(\mu) = \mu =  v(T,\mu)\;,
\eeq
for $0 > \mu > -j$, $j \geq 1$, with equality holding if $\mu =-j$. Notice that
$$
\ell_0(\mu) = \mu = \sigma_1(\mu) \;.
$$
Because of \eqref{est31} and \eqref{functeq}, and by continuity of $\mu \mapsto v(\Psi,\mu)$, we have
\beq
v(\Psi,\mu) = v(T,\mu) = \mu = \sigma_1(\mu) \; , \;\mbox{for}\;\mu \geq  -1 \;.
\eeq
We now reason by induction on $n = 1,2,\dots$. We assume that, for any $j=1,2,\dots,n$ the side of projection $[-j,1-j]$ on the $\mu$-axis of the valuation polygon of $\Psi$ is the graph of $\sigma_j (\mu)$.  This at least was proven for $n=1$. We consider the various terms in the functional equation

$$ \Psi = T - p^{-1}\Psi(pT)^q - p^{-2}\Psi(p^2T)^{q^2} - \sum_{j=3}^\infty p^{-j}\Psi(p^jT)^{q^j} \; .
$$
We assume $n >1$. For $j= 1,2,\dots,n$, and $-n-1 < \mu <-n$, we have
\beq \label{est4} \begin{split}
v(p^{-j} \Psi(p^j T)^{q^j}, \mu) = -j + v(\Psi(p^j T)^{q^j}, \mu) = -j + q^j v(\Psi(p^j T), \mu) = \\  -j + q^j v(\Psi(S), j+\mu)= -j +q^j\sigma_{n-j+1}(\mu+j)
\;,
\end{split}
\eeq
since $ j-n-1 < j+\mu < j-n$, and therefore the inductive assumption gives $v(\Psi, j+\mu)= \sigma_{n-j+1}(\mu+j)$ in that interval.
For $j>n$, and $-n-1 < \mu$, we have instead, from \eqref{est31}, $v(p^{-j} \Psi(p^j T)^{q^j}, \mu) = \ell_j(\mu)$.

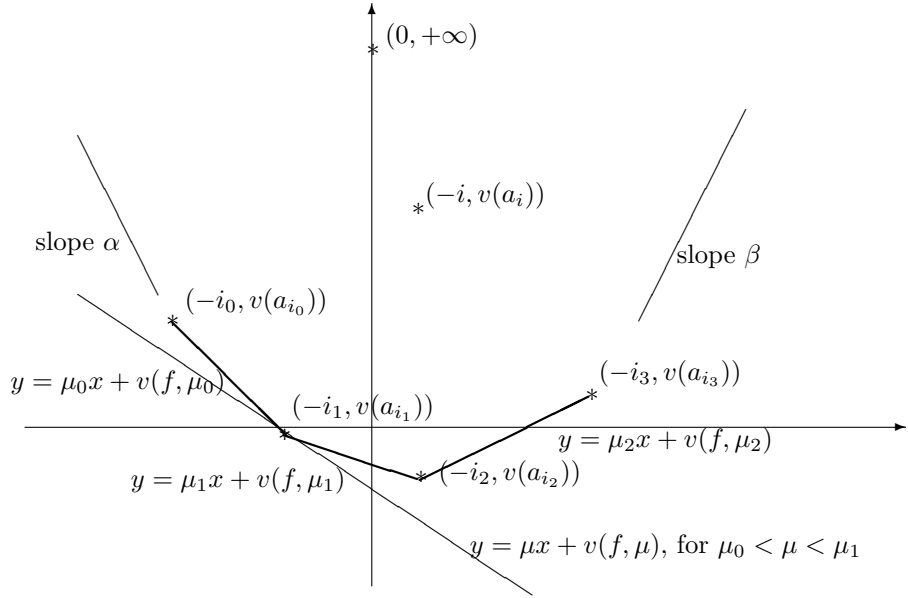
\begin{figure}[h!t] \label{figure:valuation polygon}
\begin{picture}(330,260)(-50,30)     
\put(0,100){\vector(1,0){330}}       
\put(130,40){\vector(0,1){220}}         
\put(20,210){\line(1,-2){30}}
\put(148,80){\line(2,1){63}}
\put(230,140){\line(1,2){40}}
\put(20,150){\line(3,-2){170}}
\thicklines
\put(60,145){\makebox{$(-i_0,v(a_{i_0}))$}}
\put(53,138){\makebox{$\ast$}}
\put(35,112){\makebox(0,0)[b]{$ y = \mu_0 x + v(f,\mu_0)$}}
\put(80,75){\makebox(0,0)[b]{$ y = \mu_1 x + v(f,\mu_1)$}}
\put(240,50){\makebox(0,0)[b]{$ y = \mu x + v(f,\mu)$, for $\mu_0 < \mu < \mu_1$}}
\put(240,90){\makebox(0,0)[b]{$ y = \mu_2 x + v(f,\mu_2)$}}
\put(55,140){\line(1,-1){43}}
\put(100,105){\makebox{$(-i_1,v(a_{i_1}))$}}
\put(95,95){\makebox{$\ast$}}
\put(97,97){\line(3,-1){50}}
\put(150,185){\makebox{$(-i,v(a_i))$}}
\put(145,180){\makebox{$\ast$}}
\put(155,80){\makebox{$(-i_2,v(a_{i_2}))$}}
\put(146,79){\makebox{$\ast$}}
\put(215,118){\makebox{$(-i_3,v(a_{i_3}))$}}
\put(210,110){\makebox{$\ast$}}
\put(135,245){\makebox{$(0,+ \infty)$}}
\put(128,240){\makebox{$\ast$}}
\put(20,165){\makebox(0,0)[b]{slope $\alpha$}}
\put(260,160){\makebox(0,0)[b]{slope $\beta$}}
\put(148,80){\line(2,1){63}}
\end{picture}
\caption{The Newton polygon ${\rm Nw}(f)$ of $f$.}
\end{figure}

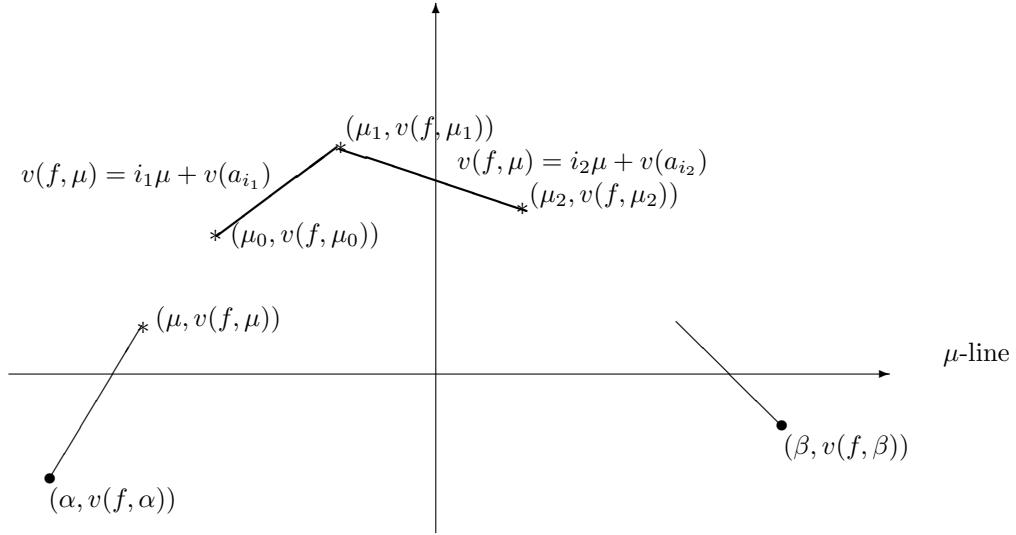
\begin{figure}[h!b] \label{figure:valpolygon}
\begin{picture}(330,200)(-50,30)     
\put(-30,100){\vector(1,0){330}}       
\put(130,40){\vector(0,1){200}}         
\put(-15,60){\line(3,5){35}}
\put(220,120){\line(1,-1){40}}
\thicklines
\put(320,105){\makebox{$\mu$-line}}
\put(25,118){\makebox{$(\mu,v(f,\mu))$}}
\put(-15,50){\makebox{$(\alpha,v(f,\alpha))$}}
\put(-17,58){\makebox{$\bullet$}}
\put(260,70){\makebox{$(\beta,v(f,\beta))$}}
\put(257,78){\makebox{$\bullet$}}
\put(18,115){\makebox{$\ast$}}
\put(53,150){\makebox{$(\mu_0,v(f,\mu_0))$}}
\put(45,150){\makebox{$\ast$}}
\put(96,190){\makebox{$(\mu_1,v(f,\mu_1))$}}
\put(92,183){\makebox{$\ast$}}
\put(94,185){\line(3,-1){70}}
\put(165,165){\makebox{$(\mu_2,v(f,\mu_2))$}}
\put(160,160){\makebox{$\ast$}}
\put(48,152){\line(4,3){45}}
\put(20,170){\makebox(0,0)[b]{ $v(f,\mu)= i_1 \mu + v(a_{i_1})$}}
\put(185,175){\makebox(0,0)[b]{$v(f,\mu) =i_2 \mu + v(a_{i_2})$}}
\end{picture}
\caption{The valuation polygon  ${\rm Val}(f)$.}
\end{figure}

\begin{lemma} Let $n>1$.
\ben
\item For $j=1,2,\dots,n$ and for any $\mu \in \R$,
\beq\label{item1}
\sigma_{n+1}(\mu) < -j +q^j\sigma_{n-j+1}(\mu+j) \;.
\eeq
\item
For $j>n$ and $\mu >-n-1$, we have
 \beq\label{item2}
 \sigma_{n+1}(\mu) < \ell_j(\mu) \;.
 \eeq
 \item
For $-n-1 <\mu <-n$,
 \beq\label{item3}
 \sigma_{n+1}(\mu) < \mu
 \;.
 \eeq
 \een
\end{lemma}
\begin{proof}
Assertion \eqref{item1} is clear, since the two affine functions $\mu \mapsto \sigma_{n+1}(\mu)$ and $\mu \mapsto -j +q^j\sigma_{n-j+1}(\mu+j)$, have the same slope $q^n$, while
their values at $\mu = -n$ are $-q^{n-1}-q^{n-2}-\dots-q-1$ and $-j -q^{n-1}-q^{n-2}-\dots-q^{j+1}-q^j$, respectively. Notice that
$$-j -q^{n-1}-q^{n-2}-\dots-q^{j+1}-q^j  = (q^j -j) -q^{n-1}-q^{n-2}-\dots-p-1 > -q^{n-1}-q^{n-2}-\dots-q-1 \;,
$$
so that the conclusion follows.
\par
We examine assertion \eqref{item2}, namely that, for $j>n$ and $\mu >-n-1$, we have
$$
q^n(\mu+n) -q^{n-1}-q^{n-2}-\dots-q-1 < - j + (j +\mu)q^j \;.
$$
The previous inequality translates into
$$
q^n(\mu+n) -q^{n-1}-q^{n-2}-\dots-q-1 < - j + (j -n)q^j +(n+\mu)q^{n+(j-n)} \;,
$$
that is
$$
 q^{n-1}+q^{n-2}+\dots+q+1 - j +  (j -n)q^j + (n+\mu)q^{n}(q^{j-n} -1) >0\;,
$$
for $\mu >-n-1$. Since the l.h.s. is an increasing function of $\mu$, it suffices to show that the inequality hold for $\mu =-n-1$, that is to prove that
\beq \label{calcoli1}
 q^{n-1}+q^{n-2}+\dots+q+1 - j +  (j -n)q^j  - q^{n}(q^{j-n} -1) >0 \;,
\eeq
for any $j >n >1$.  We rewrite the l.h.s. of \eqref{calcoli1} as
\beq \label{calcoli2}
\begin{split}
 q^{n-1}+q^{n-2}+\dots+q+1 - n + (n-j) +  (j -n)q^j  - q^j + q^n = \\  (q^{n-1}+q^{n-2}+\dots+q+1 -n) + (q^j-1)(j-n) + (q^n-q^j)\;,
 \end{split}
\eeq
where the four terms in round brackets on the r.h.s. are each, obviously, positive numbers. The conclusion follows.
\par
We finally show  \eqref{item3} , namely that for $-n-1 <\mu <-n$,
$$
q^n(\mu+n) -q^{n-1}-q^{n-2}-\dots-p-1 < \mu \;.
$$
It suffices to compare the values at $\mu =-n-1$ and at $\mu =-n$. We get
$$
- q^n  - q^{n-1}-q^{n-2}-\dots-q-1 < -n-1\;,
$$
and
$$
-q^{n-1}-q^{n-2}-\dots-q-1 < -n\;,
$$
respectively, both obviously true.
\end{proof}  
 The previous calculation shows  that the side of projection $[-n-1,-n]$ on the $\mu$-axis of the valuation polygon of $\Psi$ is the graph of $\sigma_{n+1} (\mu)$. We have then crossed  the inductive step Case $n \Rightarrow$ Case $n+1$, and Theorem~\ref{valpolpsi} is proven.
\end{proof}
\begin{cor}\label{basicest0} Proposition~\ref{unifPsi} holds true.
\end{cor}
\begin{proof} We have seen that $v_p(\Psi_p(x)) = v_p(x)$ if $v_p(x) >0$. 
Then Proposition~\ref{unifPsi}  follows from  \eqref{covsum2}  and Lemma~\ref{phiisob}. 
\end{proof} 
\begin{cor}\label{basicest} For any $i =1,2,\dots$, and $v_p(x) \geq -i $ (resp. $v_p(x) > -i$), we have $v_p(\Psi_q(x)) \geq - \frac{q^i-1}{q-1}$ (resp. $v_p(\Psi_q(x)) >- \frac{q^i-1}{q-1}$). If $v_p(x) > -1$, we have $v_p(\Psi_q(x)) = v_p(x)$.
\end{cor}
\begin{proof} The last part of the statement is a general fact for automorphisms of an open $k$-analytic disk $D$ with one $k$-rational fixed point $a \in D(k)$ (the disk $v_p(x) > -1$ and $x(a)=0$, in the present case) \cite[Lemma 6.4.4]{Berkovich}.
\end{proof}
\end{subsection}
\begin{subsection}{Newton polygon of $\Psi_q$} \label{Newtpol}
We now   recall that to a Laurent series $f = \sum_{i \in \Z} a_i T^i$ with coefficients $a_i \in \C_p$, converging in an annulus $A:=\alpha \leq v_p(T) \leq \beta$, one associates
two, dually related, polygons. The valuation polygon $\mu \mapsto v(f,\mu)$, was recalled before.
The \emph{Newton polygon} ${\rm Nw}(f)$ of $f$ is the convex closure in the standard affine plane $\R^2$ of the points $(-i,v(a_i))$ and $(0,+\infty)$.   If $a_i =0$, then $v (a_i)$ is understood as $= + \infty$. We define $s \mapsto {\rm Nw}(f,s)$ to be the function whose graph is the lower-boundary of ${\rm Nw}(f)$.  The main property of
${\rm Nw}(f)$ is that the length of the projection on the $X$-axis of the side of slope $\sigma$ is the number of zeros of $f$ of valuation $=\sigma$.
The formula
$$
v(f,\mu) = \inf_{i \in \Z} i \,\mu + v(a_i)
$$
indicates (\cf \cite{Laz}) that the relation between  ${\rm Nw}(f)$ and ${\rm Val}(f)$ ``almost'' coincides with the duality formally described in the following lemma.
\begin{lemma} \label{polar} {\bf (Duality of polygons)} In the projective plane $\P^2$,  with affine coordinates $(X,Y)$, we
consider the polarity with respect to the
parabola $X^2 = -2Y$
$$ \P^2 \to (\P^2)^\ast \to \P^2\;,
$$
$$
\mbox{\rm point}\; (\sigma, \tau)  \longmapsto \;\mbox{\rm line}\; (Y = - \sigma X - \tau) \longmapsto \; \mbox{\rm point}\;  (\sigma, \tau) \; .
$$
Assume the graph $\Gamma$ of a continuous convex piecewise affine function has consecutive vertices
$$
\dots, \, (-i_0,\varphi_0)\, ,\, (-i_1,\varphi_1)\, ,\, (-i_2,\varphi_2)\,,\, (-i_3,\varphi_3)\,,\, \dots
$$
joined by the lines
$$\dots\, ,\, Y = \sigma_1X + \tau_1\, ,\, Y = \sigma_2 X + \tau_2\, ,\, Y = \sigma_3X  + \tau_3\,, \dots \;.
$$
Then, the lines joining the points
$$ \dots \, ,\, (-\sigma_1, -\tau_1)\, ,\, (- \sigma_2, - \tau_2)\, ,\, (-\sigma_3, -\tau_3)\,,\, \dots
$$  are
$$\dots\, ,\, Y =  i_1X - \varphi_1\, ,\, Y =  i_2 X - \varphi_2\, ,\,  \dots\;,
$$
and the polarity transforms these back into
$$
\dots \,, \, (-i_1,\varphi_1)\, ,\, (-i_2,\varphi_2)\,,\, \dots \;.
$$
We say that the graph $\Gamma^\ast$ joining the vertices $(\sigma_i, \tau_i)$, $(\sigma_{i+1}, \tau_{i+1})$ by a straight segment is the \emph{dual graph} of $\Gamma$. It is clear that the relation is reciprocal, that is $(\Gamma^\ast)^\ast = \Gamma$ and that $\Gamma^\ast$ is a continuous concave piecewise affine function.
\end{lemma}
\begin{proof} It is the magic of polarities.
\end{proof}
The precise relation between ${\rm Nw}(f)$ and ${\rm Val}(f)$ is 
\begin{prop}
$${\rm Val}(f) = (-{\rm Nw}(f))^\ast
$$
where $-{\rm Nw}(f)$ is the polygon obtained from ${\rm Nw}(f)$ by the transformation $(X,Y) \mapsto (X,-Y)$.
\end{prop}
\begin{proof}
The most convincing proof follows from comparing Lemma~\ref{polar}  with Figures 2 and 3.
\end{proof}
We now apply the previous considerations to the two polygons associated to the function $\Psi_q$.
\begin{cor} The Newton polygon ${\rm Nw}(\Psi_q)$ has vertices  at the points
$$V_i:=(-q^i, i\,q^i-\frac{q^i-1}{q-1}) = (-q^i, i\,q^i-q^{i-1}-\dots-q-1)\;.$$
The equation of the side joining the vertices $V_i$  and $V_{i-1}$ is
$$
Y = -iX - \frac{q^i-1}{q-1} \; ;
$$
 its projection on the $X$-axis is the segment $[-q^i,-q^{i-1}]$.
So, ${\rm Nw}(\Psi)$ has the form described in Figure 4.
\end{cor}
\begin{figure}[h!t] \label{figure:valpolygon}
\begin{picture}(300,200)(-50,0)     
\put(0,40){\vector(1,0){270}}       
\put(240,10){\vector(0,1){180}}         
\thicklines
\put(217,40){\line(0,1){170}}
\put(174,83){\line(1,-1){42}}
\put(115,202){\line(1,-2){60}}
\put(215,38){\makebox{$\bullet$}}
\put(174,38){\makebox{$\bullet$}}
\put(30,38){\makebox{$\bullet$}}
\put(215,45){\makebox{$-1$}} 
\put(174,45){\makebox{$-q$}} 
\put(30,45){\makebox{$-q^2$}} 
\put(170,55){\makebox(0,0)[b]{slope -1}}
\put(170,145){\makebox(0,0)[b]{slope  -2}}
\put(40,130){\makebox(0,0)[b]{vertex $V_i$ at  $(-q^i, i\,q^i-q^{i-1}-\dots-q-1)$}}
\put(120,70){\makebox(0,0)[b]{vertex $V_1$ at  $(-q, q-1)$}}
\put(173,80){\makebox{$\ast$}}
\put(57,185){\makebox(0,0)[b]{vertex $V_2$ at  $(-q^2, 2q^2-q-1)$}}
\put(105,210){\makebox{$\ast$}}
\end{picture}
\caption{The Newton polygon ${\rm Nw}(\Psi_q)$ of $\Psi_q$.}
\end{figure}
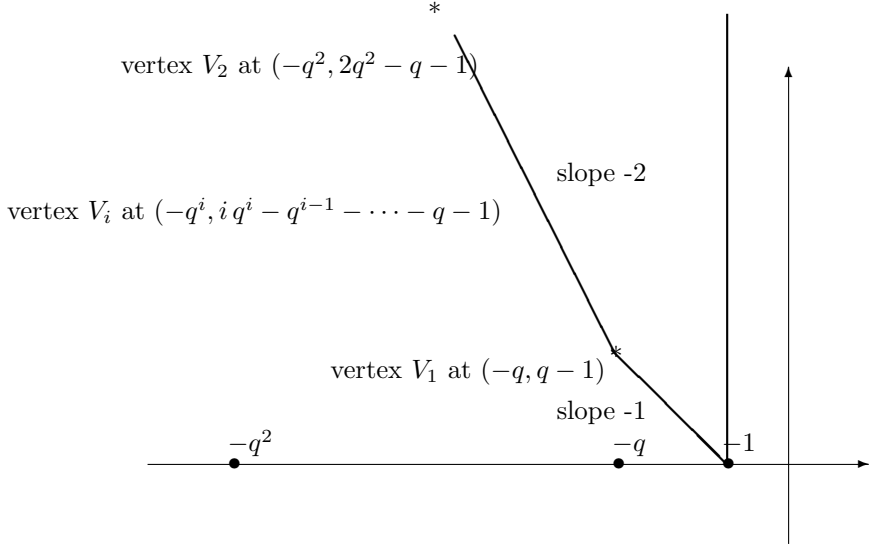
\begin{cor}  \label{etcov0} For any $i=0, 1,  \dots$,   the  map $\Psi = \Psi_q$ induces  coverings of degree $q^i$,
\beq \label{etcov}
\Psi :  \{\,x \in \C_p\,| \, v_p(x) >  -i -1 \,\} \longrightarrow \{\,x \in \C_p\,| \, v_p(x) >  -\frac{q^{i+1}-1}{q-1}   \,\} \; ,
\eeq
(in particular, an isomorphism
\beq \label{locisom}
\Psi :  \{\,x \in \C_p\,| \, v_p(x) > -1 \,\} \iso \{\,x \in \C_p\,| \, v_p(x) > -1  \,\} \;,
\eeq
for $i=0$),
finite  maps of degree $q^i$
\beq \label{etcov1}
\Psi :  \{\,x \in \C_p\,| \, -(i+1)< v_p(x) < -i \,\} \longrightarrow \{\,x \in \C_p\,| \, -\frac{q^{i+1}-1}{q-1} < v_p(x) <  -\frac{q^{i}-1}{q-1}   \,\} \; ,
\eeq
and finite  maps of degree $q^{i+1} - q^i$
\beq \label{etcov3}
\Psi :  \{\,x \in \C_p\,| \,  v_p(x) = -i - 1\,\} \longrightarrow \{\,x \in \C_p\,| \, -\frac{q^{i +1}-1}{q-1} \leq v_p(x)   \,\} \; .
\eeq
\end{cor} 
\begin{proof} The shape of the Newton polygon of $\Psi$ indicates that, for any $a \in \C_p$, with $v_p(a) >-1$, the side of slope $= v_p(a)$ of the
Newton polygon of $\Psi -a$ has projection of length 1 on the $X$-axis. So, $\Psi : \{\,x \in \C_p\,| \, v_p(x) > -1\,\} \to \{\,x \in \C_p\,| \, v_p(x) > -1\,\}$ is bijective, hence biholomorphic.   Now we recall from Corollary~\ref{basicest} that for any given $i \geq 1$,
\beq \label{etcov2}
\Psi ( \{\,x \in \C_p\,| \, v_p(x) >  -i -1 \,\}) \subset  \{\,x \in \C_p\,| \, v_p(x) >  -\frac{q^{i+1}-1}{q-1}   \,\} \; .
\eeq
So, let $a$ be such that  $-\frac{q^{i+1}-1}{q-1} < v_p(a) \leq -\frac{q^{i}-1}{q-1}$, say $v_p(a) = -\frac{q^{i}-1}{q-1} -\veps$, with $\veps \in [0,q^i)$.
Then, the
Newton polygon of $\Psi -a$ has a single side of slope $>  \, -i -1$, which has precisely slope $= -\veps \, q^{-i}-i$ and  has projection of length $q^i$ on the $X$-axis. So, the equation
$\Psi(x) =a$ has precisely $q^i$ solutions $x$ in the  annulus $-i  -1 <  v_p(x) \leq  \, -i$.   If, for the same $i$,  $-\frac{q^{i}-1}{q-1} < v_p(a) \leq -\frac{q^{i -1}-1}{q-1}$,  the
Newton polygon of $\Psi -a$ has a  side of slope $-i$, whose projection  on the $X$-axis has  length $q^i -q^{i-1}$, and a side of slope $\sigma$, $1-i \geq \sigma >-i$, whose projection  on the $X$-axis has  length $q^{i-1}$. So again $\Psi^{-1}(a)$ consists of $q^i$ distinct points.
We go on,  for $a$ in an annulus of the form $-\frac{q^{i-j}-1}{q-1} < v_p(a) \leq -\frac{q^{i -j -1}-1}{q-1}$, up to  $j=i-2$, \ie to $-\frac{q^2 -1}{q-1} < v_p(a) \leq - 1$.  In that case, the
Newton polygon of $\Psi -a$ has a  side of slope $-i$ of projection $q^i -q^{i-1}$, a  side of slope $1-i$ of projection $q^{i-1} - q^{i-1}$,\dots,  a  side of slope $j-i$ of projection $q^{i-j} - q^{i-j-1}$ on the $X$-axis, \dots,  up to a side of slope -1 of projection $q-1$ on the $X$-axis. Finally, for $v_p(a)>-1$, there is still exactly one solution of $\Psi(x) =a$, with $v_p(x) >-1$.
This means that $\Psi$  induces a (ramified) covering of degree $q^i$ in \eqref{etcov}.
\end{proof}
\end{subsection}
\begin{subsection}{The addition law of  $\Psi$}
We now extend  the estimates of Corollary~\ref{etcov0} to translates $\Psi (a + x)$ of $\Psi$, for $a \in \Q_p$. Although we expect that the same discussion carries over to $\Psi_q(a + x)$, where $a \in \Q_q$, we assume for simplicity that 
 $q=p$ in the rest of this subsection.
\begin{prop} \label{estadm}    Let $m \in \Z_{>0}$ and let  $M(x_{-m} \dots,x_{-1},x_0)$ be a monomial
in  $\Z_p[x_{-m} \dots,x_{-1},x_0]$ divisible by $x_{-m}$ and of pure weight $1$, where $x_i$ weighs $p^i$, for any $i$. 
Set
$$
M(x) := M(\Psi(p^m x), \dots, \Psi(x)) \;.
$$
Then, for any $r=1,2,\dots$,
 \beq \label{r=1,2}
\begin{split}
w_r(M(x))& \geq  \\  
m  +1 - (p-1)r (m-r+1) + (p-1) & \l({{m+1} \choose {2}} - {{r} \choose {2}}\r) - \frac{p^{r  +1}-1}{p-1} >   \\
m  +1 + (p-1) \frac{(m-r)^2 + (m-r)}{2}& - \frac{p^{r  +1}-1}{p-1} \; (\,> - \frac{p^{r  +1}-1}{p-1} \,)
\;,
\end{split}
\eeq
while, for $r =0,-1,-2,\dots$, we get 
 \beq \label{r=0,-1}
 w_r(M(x)) \geq m-r  - (p-1)mr + (p-1) {{m+1} \choose {2}} \;(\, \geq p \,(1-r) \,)\; .
\eeq
\end{prop}
\begin{proof} This follows from the estimates of Corollary~\ref{etcov0} via  a  totally 
 self-contained, but lengthy, computation on isobaric polynomials of Witt-type.  
We refer to the upcoming paper \cite{perf_fourier} for the   proof of a more general statement. 
\end{proof}
We apply Proposition~\ref{estadm}  to the study of the addition law of $\Psi$. From \eqref{covsum2}  and 
\eqref{covsum3}, we deduce, taking into account Proposition~\ref{wittdecomp}, that, for any $c \in \Q_p$
\beq
\label{covsum4}  \begin{split}
 &\Psi(x + c)   = \\ \lim_{i \to \infty} 
 \varphi_i(\Psi(p^ix), \dots,&\Psi(px),\Psi(x);\Psi(p^i c),\dots,\Psi(p c),\Psi(c)) \;,
 \end{split}
\eeq
where 
\beq
\label{covsum5}  \begin{split}
\varphi_m(\Psi(p^m x),& \dots,\Psi(px),\Psi(x);\Psi(p^m c),\dots,\Psi(p c),\Psi(c)) - \\
& \varphi_{m-1}(\Psi(p^{m-1}x), \dots,\Psi(px),\Psi(x);\Psi(p^{m-1} c),\dots,\Psi(p c),\Psi(c)) 
\end{split}
\eeq
is  a sum of monomials $M(x)$ as in Proposition~\ref{estadm}.  
  \begin{thm} \label{psibddsuper} \hfill
  \ben
  \item
  The function $\Psi$ is bounded and uniformly continuous on any $p$-adic strip around $\Q_p$. In particular, 
  $$
  \Psi(x) \in  \sH^{\bd} \;.
  $$
  \item For any $j =0,1,\dots$ and $x \in \Q_p$,
\beq
\label{unifPsi11}
\Psi_p(x + p^j  \C_p^\circ ) \subset \Psi_p(x) + p^j  \C_p^\circ \;.
\eeq
\een
   \end{thm}
   \begin{proof}
For the first part of the statement, we observe that  Proposition~\ref{estadm} shows that, for  any fixed $r \geq 0$, 
 the sequence 
$$
i \longmapsto  \varphi_i(\Psi(p^ix), \dots,\Psi(px),\Psi(x);\Psi(p^i c),\dots,\Psi(p c),\Psi(c))
$$
converges in the $w_r$-valuation. This means that for
 any $\rho >0$, the previous sequence 
is a sequence of entire functions bounded on the $p$-adic strip $\Sigma_\rho$ around $\Q_p$,  which  converges to $\Psi(x + c)$ uniformly on $\Sigma_\rho$. 
\par
The second part of the statement was already proved in Corollary~\ref{basicest0}. It also follows from the estimates of Proposition~\ref{estadm} when $r \leq 0$. 
\end{proof}
\end{subsection}
\begin{subsection}{The zeros of $\Psi$} \label{zeropsi}
The following theorem is formulated in a way to make sense for $q =$ any power of $p$. We expect that it is true 
in that generality. However, for the time being, we can only prove it for $q = p$.
\begin{thm} \label{invfunct} In this statement, let $q=p$. 
\ben
\item
For any $n=1,2,\dots$, the map $\Psi_q$ has $q_n := q^n-q^{n-1}$ simple zeros of valuation $-n$ in $\Q_q$. More precisely, for any system of representatives $a_1, \dots, a_{q_n} \in \Z_q$ of $(\Z_q/p^n \Z_q)^\times = {\rm W}_{n-1}(\F_q)^\times$, and any $j =1,\dots,q_n$, the open disc $D( a_j p^{-n},
p^-)$ 
 contains a unique zero $z^{(n)}_j \in \Q_q$ of $\Psi_q$. Then  $z^{(n)}_1,\dots,z^{(n)}_{q_n}$ are all the zeros of $\Psi_q$ of valuation $-n$. 
 \item
For $n=1,2,\dots$ let $z^{(n)}_1,\dots,z^{(n)}_{q_n}$  be the zeros of $\Psi_q$ of valuation $-n$. 
We set
$$\psi_n(x) = \prod_{j=1}^{q_n} (1- \frac{x}{z^{(n)}_j}) \in 1 + p^n x\Z_q[x]
\; .$$
Then 
\beq \label{prodformPsi}
\Psi_q(x) = x \prod_{n=1}^\infty \psi_n(x) 
\eeq 
is the canonical  convergent infinite Schnirelmann product expression \cite[(4.13)]{Laz} of $\Psi_q(x)$ in the ring $\Q_p\{x\}$. 
\item
The inverse function $\beta(T) = \beta_q(T)$ of $\Psi_q(T)$ (\ie the power series such that, in $T\Z[[T]]$, $\Psi_q(\beta_q(T))=T = \beta_q(\Psi_q(T))$) belongs to $T+T^2\Z[[T]]$. Its disc of convergence is exactly $v_p(T)>-1$.
\een
\end{thm} 
\begin{proof}  We now prove  the first statement in Theorem~\ref{invfunct}.  We recall that here $q=p$, so that $a_1, \dots, a_{p_n} \in \Z_p$, with $p_n = p^n - p^{n-1}$, are a system of representatives of $(\Z_p/p^n \Z_p)^\times$. 
 \begin{lemma} \label{diffest} For any $m,n \in \Z_{>0}$, with $m \leq n$, and any $j = 1,\dots, p_n$,  the value of $\Psi$ at the maximal point $\xi_{a_j p^{-n},p^{m}}$ (of Berkovich type 2) of the rigid 
disc $D(a_j p^{-n}, (p^{m})^+)$, that is $-\log |\Psi(\xi_{a_j p^{-n},p^{m}})| = w_m(\Psi (a_j p^{-n} +x))$,  is $-\frac{p^{m}-1}{p-1} <0$.
\end{lemma}  
\begin{proof} The proof   follows from the addition law \eqref{covsum4} in which 
$$\Psi(a_j p^{-n}), \dots,  \Psi(a_j p^{i-n}) \in \Z_p 
$$
so that, for any $i =0,1,2,\dots$,
$$
 \varphi_i(\Psi(p^ix), \dots,\Psi(px),\Psi(x);\Psi(a_j p^{i-n}),\dots,\Psi(a_j p^{1-n}),\Psi(a_j p^{-n})) 
 $$
is a sum of a dominant (at $\xi_{a_j p^{-n},p^{m}}$)
term 
$$
\Psi(x) + \Psi(a_j p^{-n})
$$
and 
of terms $M(x)$  described, for $m = n-i$, in Proposition~\ref{estadm}.
\end{proof}
 From the  harmonicity  of the function $|\Psi(x)$ at  the point 
$\xi_{a_j p^{-n},p^{m}}$, the estimate of Lemma~\ref{diffest}, and the fact 
 that $\Psi(a_j p^{-n}) \in \Z_p$, we deduce   that each of the $p_n p^{m-n}$ open discs of radius $p^{m}$, centered at points of $p^{-n} \Z_p \, \setminus\,  p^{1-n} \Z_p$  
contains at least one zero of $\Psi_q$ in $\Q_q$. For $m=1$, this proves the first part of the statement.
\par
For the second part of the statement we refer to \cite[\S 4]{Laz}. The fact that every $\psi_n(x) \in \Q_p[x]$ is $-n$-extremal follows from the fact that its zeros are all of exact valuation $-n$ \cite[(2.7')]{Laz}. 
\par
The fact that $\beta_p$ belongs to $T+T^2\Z[[T]]$ is obvious. The convergence of $\beta_q$ for $v_p(T)>-1$ follows from \eqref{locisom}. The fact that it cannot converge in a bigger disk
is a consequence of the fact that $\Psi_q$ has $q-1$  zeros of valuation $-1$.
\end{proof}
\begin{cor} \label{zerodistr} All zeros of $\Psi_q$ are simple and are contained in $\Q_q$. Each ball $a + \Z_q \in \Q_q/\Z_q$ contains a single zero of $\Psi_q$.
\end{cor}
\begin{rmk} \label{invfunct2} We believe that Theorem~\ref{invfunct} holds, with essentially the same proof,   for any power $q$ of $p$.  See Remark~\ref{genq}.
\end{rmk} 
\end{subsection}
 \begin{section}{Rings of continuous  functions on $\Q_p$} \label{contfcts} 
 The point of this section is that of establishing the categorical limit/colimit formulas for the linear topologies of rings of $p$-adic functions on $\Q_p$. For topological algebra notions, we take the viewpoint and use the definitions explained in Appendix A (see also \cite{closed}).
 \par \smallskip
 We consider here a linearly topologized separated and complete ring $k$, whose family of open ideals we denote by $\cP(k)$. In practice $k = \Z_p$ or $=\F_p$, or $=\Z_p/p^r\Z_p$, for any $r \in \Z_{\geq 1}$.  More generally, $A$ will be a  complete and separated  topological ring equipped with a $\Z$-linear  topology, defined by a family of open additive subgroups of $A$. In particular we have in mind $A=$  a fixed finite extension $K$ of $\Q_p$, whose topology is $K^\circ$-linear but not $K$-linear. Again, a possible $k$ would be $K^\circ$ or any $K^\circ/(\pi_K)^r$, for a parameter $\pi = \pi_K$ of $K$, and $r$ as before. 
 \par 
 We will  express our statements  for an abelian topological group $G$, which is separated and complete in the $\Z$-linear topology
 defined by a countable family of profinite subgroups $G_r$, with $G_r \supset G_{r+1}$, for any $r \in \Z$. So, 
$$G = \limPROu_{r \to +\infty} G/G_r = \limINDu_{r \to -\infty} G_r
\;,
$$ 
where $G/G_r$ is discrete, $G_r$ is compact, and limits and colimits are taken in the category of topological abelian groups separated and complete in a $\Z$-linear topology. 
 We denote by  $\pi_r : G \to G/G_r$  the canonical projection. 
 Then, $G$ is canonically a uniform space in which a function $f:G \to A$ is uniformly continuous iff, for any open subgroup $J \subset A$, the induced function $G \to A/J$  factors via a $\pi_r$, for some $r = r(J)$. A subset of $G$ of the form 
 $\pi_r^{-1}(\{h\}) = g + G_r$, for $g \in G$ and $h = \pi_r(g)$   is sometimes called the  \emph{ball of radius $G_r$} and \emph{center} $g$. 
 In particular, $G$ is a locally compact, paracompact, $0$-dimensional topological space.  A general discussion of the duality between $k$-valued functions and measures on such a space, will appear in \cite{duality}. 
In practice here $G= \Q_p$ or $\Q_p/p^r \Z_p$ or $p^r \Z_p$, with the obvious uniform and topological structure.  
 \begin{defn} \label{contdef}
 Let $G$ and $A$ be as before. We define $\sC(G,A)$ (resp. $\sC^\bd_\unif(G,A)$) as the $A$-algebra of continuous (resp. bounded and uniformly continuous)  functions $f : G \to A$. 
 We equip $\sC(G,A)$ (resp.   $\sC^\bd_\unif(G,A)$)  with the topology of uniform convergence on compact subsets of $X$ (resp. on $X$). For any  $r \in \Z$ and $g \in G$, we denote by $\chi_{g+G_r}$ is the characteristic function of $g+G_r \in G/G_r$. If $G$ is discrete and $h \in G$, by $e_h:G \to k$ we denote the function such that $e_h(h) =1$, while $e_h(x) =0$ for any $x \neq h$ in $G$.
 \end{defn} 
 \begin{rmk} It is clear that if $A=k$ is a linearly topologized ring any subset of $k$ and therefore any function $f:G \to k$, is bounded. So, we write $\sC_\unif(G,k)$ instead of $\sC^\bd_\unif(G,k)$ in this case. If $G$ is discrete, any 
 function $G \to k$ is (uniformly) continuous; still, 
the bijective map  $\sC_\unif(G,k) \to \sC(G,k)$ is not an isomorphism in general, so we do keep 
the difference in notation. If $G$ is compact, any continuous function $G \to k$ is uniformly continuous, and
$\sC_\unif(G,k) \to \sC(G,k)$ is an isomorphism, so there is no need to make any distinction.
 \end{rmk}

 \medskip
 \begin{lemma} \label{discretecase} Notation as above, but assume $G$ is discrete (so that the $G_r$'s are finite). Then $\sC(G,k)$ (resp.  $\sC_\unif(G,k)$) is the 
$k$-module of functions $f:G \to k$ endowed with the topology of simple (resp. of uniform) convergence on $G$. So
 $$
 \sC(G,k) =  \limPROu_{r \to -\infty} \sC(G_r,k)= \prod_{h} k \,e_h \;\;,\;\;h \in G \;.
 $$ 
Similarly, 
 $$
 \sC_\unif(G,k) =    \limPROu_{I \in \cP(k)}  \PRODSCU{h \in G}\, (k/I) \,e_h =  \PRODSCU{h \in G}\, k \,e_h  \;,
 $$  
 where $\dspl \PRODSCU{h \in G} \, (k/I) \,e_h$ carries the discrete topology. 
 \end{lemma}
 \begin{proof} Clear from the definitions.
 \end{proof}
 The next lemma is a simplified  abstract form, in the framework of linearly topologized rings and modules,  of the classical  decomposition of a continuous function as a sum of characteristic functions of balls 
 (see for example Colmez \cite[\S 1.3.1]{colmez2}).
\begin{lemma} \label{compactcase} Notation as above but assume $G$ is compact (so that the $G/G_r$'s are finite). Then
 $$
 \sC(G,k) =\sC_\unif(G,k)   = \limINDU_{r \to +\infty} \sC(G/G_r,k) = \limINDU_{r \to +\infty} \bigoplus_{g+ G_r \in G/G_r} k \chi_{g + G_r} \;.
 $$ 
 For any $r$, the canonical morphism  $\sC(G/G_r,k) \to  \sC(G,k)$ is injective. 
 \end{lemma} 
 \begin{proof} This is also clear from the definitions.
  \end{proof}
  \begin{rmk} \label{wavelet} We observe that the inductive limit appearing in the formula hides the
  complication of formulas of the type
  $$
\chi_{g + G_r} = \sum_{i} \chi_{g_i + G_{r+1}}\;\;\mbox{if}\;\; g + G_r = \dot\bigcup_i \; g_i + G_{r+1} 
  $$
  which we do not need to make explicit for the present use (see \cite{duality} for a detailed discussion). 
\end{rmk}
 \begin{prop} \label{contvsunif} Notation as above, with $G$ general.   Then in the category $\LMu_k$ we have~:
\ben
\item
$$ \sC(G,k) =  \limPROu_{r \to -\infty} \sC(G_r,k) \;\;\mbox{for the restrictions} \;\; \sC(G_r,k) \to \sC(G_{r+1},k) \;.
$$ 
In particular, for any fixed $r \in \Z$,
$$
 \sC(G,k) =   \prod_{g+G_r \in G/G_r} \sC(g+G_r,k) \;.
  $$   
\item
$$
\sC_\unif(G,k)  =  
 \limINDU_{r \to +\infty} \sC_\unif(G/G_r,k) 
 $$
 for the embeddings
 $$
 \sC_\unif(G/G_r,k) \hookrightarrow  \sC_\unif(G/G_{r+1},k)  
$$
\item 
The natural morphism 
$$\sC_\unif(G,k) \longrightarrow \sC(G,k)
$$
is injective and has dense image. 
\een
\end{prop}
\begin{proof} The first two parts follow from the universal properties of limits and colimits. 
The morphism in part 3 comes from the injective morphisms, for $r \in \Z$, 
$$\sC_\unif(G/G_r,k) \longrightarrow \sC(G,k)$$
and the universal property of colimits. The  inductive limit of these morphisms in the category $\LMu_k$ is a completion 
of the inductive limit taken in the category $\Mod_k$ of $k$-modules equipped with the $k$-linear inductive limit topology. 
Since the latter is separated and since the axiom AB5 holds for the abelian category $\Mod_k$, we deduce that 
the morphism in part 3 is injective.  
The morphism has dense image because, for any $r \in \Z$ and for any $s \in \Z_{\geq 0}$, the composed morphism
$$\sC_\unif(G_r/G_{r+s},k) \longrightarrow \sC_\unif(G/G_{r+s},k) \longrightarrow \sC(G,k)  \longrightarrow \sC(G_r,k)$$
is the canonical map of Lemma~\ref{compactcase}
$$\sC_\unif(G_r/G_{r+s},k) \longrightarrow \sC(G_r,k)$$
for the compact group $G_r$ and its subgroup $G_{r+s}$. The fact that the set theoretic union 
$\bigcup_{s \geq 0} \sC_\unif(G_r/G_{r+s},k)$ is dense in $\sC(G_r,k)$ is built-in in the definition of $\limINDU$. 
\end{proof}
 \begin{prop} \label{prodfcts} 
 Let  $(H,\{H_r\}_r)$ be a locally compact group with the same properties as   $(G,\{G_r\}_r)$ above, so that 
 $(G \times H,\{G_r \times H_r\}_r)$ also has the same properties. 
 Then 
  we have a natural identification  in $\LMu_k$
 \beq \label{canisomprod1}
 \sC(G,k) \wt^u_k  \sC(H,k) \iso  \sC(G\times H,k) 
 \eeq
 and a continuous strictly closed embedding
 \beq \label{canisomprod2}
 \sC_\unif(G,k) \wt^u_k  \sC_\unif(H,k) \longrightarrow  \sC_\unif(G\times H,k) \;.
 \eeq 
   \end{prop}
   \begin{proof} We prove \eqref{canisomprod1} first. By the first part of point $1$ in Proposition~\ref{contvsunif} and the fact that $\wt_k^u$ commutes with projective limits, we are reduced to the case of $G$ and $H$ compact. We are then in the situation of Lemma~\ref{compactcase} for both $G$ and $H$ (in particular, the $G/G_r$'s and the $H/H_r$'s are finite). 
 We need to prove
 $$
    \limINDU_{r \to +\infty} \bigoplus_{(g,h)+(G_r \times H_r)} k \chi_{(g,h) + G_r \times H_r}
 =   \limINDU_{r \to +\infty} \bigoplus_{g+ G_r  } k \chi_{g + G_r}
 \; \wt^u_k \;
 \limINDU_{r \to +\infty} \bigoplus_{h+ H_r  } k \chi_{h + H_r}
  \;.
   $$
   Let $M$ (resp. $N$) be the  l.h.s. (resp.  the r.h.s.)  in the previous equation. Then 
   $$
   M = \limPROu_{I \in \cP(k)} M/\ol{IM} \;\;,\;\;  N = \limPROu_{I \in \cP(k)} N/\ol{IN} \;.
   $$
   We show that $M/\ol{IM}  \iso N/\ol{IN}$, for any $I \in \cP(k)$. Now, 
   $$M/\ol{IM}  = \limind_r \bigoplus_{(g,h)+(G_r \times H_r)} (k/I) \chi_{(g,h) + G_r \times H_r} \;.
   $$
Let 
$$P:=  \limINDU_r \bigoplus_{g+ G_r  } k \chi_{g + G_r} \;,\; Q := \limINDU_r \bigoplus_{h+ H_r  } k \chi_{h + H_r} \;.
$$ 
Then 
   $$
   N/\ol{IN} =   P/\ol{IP} \otimes_{k/I} Q/\ol{IQ} =  \limind_r \bigoplus_{g+ G_r  } (k/I) \chi_{g + G_r}
 \; \otimes_{k/I} \;
 \limind_r \bigoplus_{h+ H_r  } (k/I) \chi_{h + H_r} =  M/\ol{IM} \;.
   $$
   This concludes the proof of \eqref{canisomprod1}.
    \endgraf
    We now pass to \eqref{canisomprod2}. We use formula $2$ of Proposition~\ref{contvsunif}, to replace the map in the 
    statement by 
        $$ 
         \limINDU_r \sC_\unif(G/G_r,k) \wt^u_k  \limINDU_{r \to +\infty} \sC_\unif(H/H_r,k)  \longrightarrow 
          \limINDU_r \sC_\unif((G \times H)/(G_r \times H_r),k)
 \;.
 $$
 By Lemma~\ref{discretecase} this reduces to considering 
 \beqa \begin{split} 
 &\limINDU_r  \limPROu_{I \in \cP(k)} \PRODSCU{g \in G}\, (k/I) \,e_{g+G_r} \wt^u_k  \limINDU_{r \to +\infty} \limPROu_{I \in \cP(k)} \PRODSCU{h \in H}\, (k/I) \,e_{h + H_r}
 \longrightarrow 
  \\
 &\\
 & \limINDU_r   \limPROu_{I \in \cP(k)}  \PRODSCU{(g,h)}\, (k/I) \,e_{(g+G_r,h+H_r)} \end{split}
\eeqa
As before, let $M$ (resp. $N$) be the  l.h.s. (resp.  the r.h.s.)  in the previous equation. Then 
   $$
   M = \limPROu_{I \in \cP(k)} M/\ol{IM} \;\;,\;\;  N = \limPROu_{I \in \cP(k)} N/\ol{IN} \;.
   $$
   We show that $M/\ol{IM}  \hookrightarrow N/\ol{IN}$ in an embedding with the relative topology, for any $I \in \cP(k)$. Now, 
   \beqa \begin{split} 
   M/\ol{IM}  = &\limINDU_r  \prod_{g}\, (k/I) \,e_{g+G_r}   \wt^u_{k/I} \limINDU_r  \prod_{h}\, (k/I) \,e_{h+H_r} =\\
    &\limINDU_r ( \prod_{g}\, (k/I) \,e_{g+G_r}   \wt^u_{k/I}  \prod_{h}\, (k/I) \,e_{h+H_r})  \;,   \end{split}
\eeqa
and
  $$N/\ol{IN}  = \limINDU_r  \prod_{(g,h)}\, (k/I) \,e_{(g+G_r,h+H_r)}   
   $$
where $\limINDU$ and $\wt^u_{k/I}$ are taken in the category $\LMu_{k/I}$. So, our statement is reduced 
to the fact that, for $k$, $G$, and $H$, discrete, if $k^G$ (resp. $k^H$, resp. $k^{G \times H}$) indicates 
the $k$-algebra  of functions $G \to k$ (resp. $H \to k$, resp. $G\times H \to k$) with the discrete topology, we have an inclusion
$$
k^G \otimes_k k^H \hookrightarrow k^{G \times H}\;.
$$
   \end{proof}
We are especially interested in 
  \begin{cor} \label{tesiunif3}\hfill
   \ben 
 \item  For any $r \in \Z$,  
  \beq
   \label{tesiunif31}
   \sC_\unif(\Q_p/p^{r+1}\Z_p,\Z_p) = \limPROu_{s \to +\infty}\sC(\Q_p/p^{r+1}\Z_p,\Z_p/p^s\Z_p)
  \eeq
  where $\sC(\Q_p/p^{r+1}\Z_p,\Z_p/p^s\Z_p)$ is equipped with the discrete topology. 
  It  is the $\Z_p$-algebra 
 of all maps $\Q_p/p^{r+1}\Z_p \to \Z_p$
equipped with the   $p$-adic topology;
\item  For any $r \in \Z$,  
  \beq
   \label{tesiunif31b}
   \sC(\Q_p/p^{r+1}\Z_p,\Z_p) = \limPROu_{s,t \to +\infty}\sC(p^{-t}\Z_p/p^{r+1}\Z_p,\Z_p/p^s\Z_p)
  \eeq
  where $\sC(p^{-t}\Z_p/p^{r+1}\Z_p,\Z_p/p^s\Z_p)$ is equipped with the discrete topology. 
  It  is the $\Z_p$-Hopf algebra 
 of all maps $\Q_p/p^{r+1}\Z_p \to \Z_p$
equipped with the topology of simple convergence on $\Q_p/p^{r+1}\Z_p$ for the $p$-adic topology of $\Z_p$;
 \item
  \beq 
   \label{tesiunif32}
   \sC_\unif(\Q_p,\Z_p) =  \limINDU_{r \to +\infty} \sC_\unif(\Q_p/p^{r+1}\Z_p,\Z_p) \;;
  \eeq
   \item
    \beq 
   \label{tesiunif32b}
   \sC(\Q_p,\Z_p) =  
   \limPROu_{s \to +\infty}\sC(\Q_p,\Z_p/p^s\Z_p) \;.
  \eeq
\een
  \end{cor}
  \begin{rmk}\label{conthopf}
  Formula \ref{canisomprod1} shows that $\sC(\Q_p,\Z_p)$ is a Hopf algebra object in $\LMu_{\Z_p}$.
  \end{rmk}
  \begin{rmk} \label{Barsper} We point out a tautological, but useful, formula which holds in  $\sC(\Q_p/p^{r+1}\Z_p,\F_p)$.
    For any $h \in \Q_p/p^{r+1}\Z_p$, let $e_h$ denote as before the function $\Q_p/p^{r+1}\Z_p \to \F_p$ such that 
  $e_h(h) =1$ while $e_h(x)=0$, if $x \in \Q_p/p^{r+1}$, $x \neq h$. For any $i \leq r$, the function 
  $$
  x_i : \Q_p/p^{r+1}\Z_p \longrightarrow  \F_p \;,
  $$
  was  introduced in \eqref{wittcomp3}. 
  We then have  
 \beq \label{xfcts}
 x_i  = \sum_{h \in \Q_p/p^{r+1}\Z_p} h_{i} e_h \;,
 \eeq
 where $h_i = x_i(h)$.
 \end{rmk}  
\begin{lemma}\label{discreteBanach}  Let $G$ and $K$ be as above but assume $G$ is discrete.    
Then in the category $\CLC_K$  
\ben
\item 
$$ \sC(G,K) = \prod_{g \in G} K e_g
$$ 
 is a Fr\'echet $K$-algebra.  
\item
$$
\sC^\bd_\unif(G,K)  = \ell_\infty(G,K)
$$
is the Banach $K$-algebra of bounded sequences $(a_g)_{g \in G}$ of elements of $K$, equipped with the componentwise sum and product and with the supnorm. 
\een
\end{lemma}
\begin{proof} Obvious from the definitions. 
\end{proof}
\begin{lemma} \label{compactBanach} Let $G$ and $K$ be as above, but assume $G$ is compact.    
Then in the category $\CLC_K$  
$$ \sC(G,K) =  \sC_\unif^\bd(G,K) = \ell^0_\infty(G,K)
$$ 
is the Banach $K$-algebra of sequences $(a_g)_{g \in G}$, with $a_g \in K$, such that $a_g \to 0$ along the filter of 
cofinite subsets of $G$, equipped with componentwise sum and product and with the supnorm. 
\end{lemma}
\begin{proof} This is a straightforward generalization of the classical wavelet decomposition. See \cite[Prop. 1.16]{colmez2}.  
\end{proof}
 \begin{prop} \label{contvsunifField} Let $G$ and $K$ be as in all this section.    Then in the category $\CLC_K$ we have~:
\ben
\item
$$ \sC(G,K) =  \limPROu_{r \to -\infty} \sC(G_r,K) \;\;\mbox{for the restrictions} \;\; \sC(G_r,K) \to \sC(G_{r+1},K) \;.
$$ 
In particular, $\sC(G,K)$ is a  Fr\'echet 
$K$-algebra. 
\item
$$
\sC^\bd_\unif(G,K)  = 
 \limINDu_{r \to +\infty} \sC^\bd_\unif(G/G_r,K)
 $$
 for the embeddings
 $$\sC^\bd_\unif(G/G_r,K) \hookrightarrow  \sC^\bd_\unif(G/G_{r+1},K)  \;,
$$
where the inductive limit of Banach $K$-algebras is strict. In particular, $\sC^\bd_\unif(G,K)$ is a complete bornological $K$-algebra. 
\item 
The natural morphism 
$$\sC^\bd_\unif(G,K) \longrightarrow \sC(G,K)
$$
is injective and has dense image. 
\een
\end{prop}
\begin{proof} It is clear. For the notion of a bornological topological $K$-vector space we refer to \cite[\S 6]{schneider}; the fact that the notion is stable by strict inductive limits is Example 3 on page 39 of \lc.  The statement on completeness is proved in \cite[Lemma 7.9]{schneider}.
\end{proof}
\begin{prop}\label{prodfcts2} Let $G$ and $H$ be locally compact groups as in Proposition~\ref{prodfcts}.
Then
 \beq \label{canisomprod21}
 \sC(G,K) \wt_{\pi,K} \sC(H,K) \iso  \sC(G\times H,K) \;,
 \eeq 
while the canonical map
  \beq \label{canisomprod22}
 \sC_\unif^\bd(G,K) \wt_{\pi,K} \sC_\unif^\bd(H,K) \longrightarrow  \sC_\unif^\bd(G\times H,K)  
 \eeq 
 is a strictly closed embedding of complete bornological algebras. 
 \end{prop}
 \begin{proof}
 In the case of $G$ and $H$ compact this is detailed  in the Example after Prop. 17.10 of \cite{schneider}. 
 In the general case \eqref{canisomprod21} follows by taking projective limits. 
 The statement for $\sC^\bd_\unif(G,K)$ reduces instead to \eqref{canisomprod2}. 
  \end{proof}
  \endgraf
  We point out  that $(\CLC_K,\wt_{\pi,K})$ is a $K$-linear symmetric monoidal category.
 \smallskip
  From Remarks~\ref{prodfcts} and \ref{prodfcts2}, we conclude 
 \begin{prop}\label{Hopf} Let $G$  be as in Definition~\ref{contdef}, and let $A$ be either $k$ or $K$, as before, and $\sC(G,A)$ be as in \lc. We regard $(\LMu_k,\wt_k)$ and $(\CLC_K,\wt_{\pi,K})$ as symmetric monoidal categories.  
 The coproduct, counit, and inversion
 $$
 \P(f) (x,y) = f(x+y) \;\;,\;\; \veps(f) = f(0_G)\;\;,\;\; \rho(f)(x) = f(-x) \;,
 $$
 for any $f \in \sC(G,A)$ and any $x,y \in G$, define a structure of topological 
 $A$-Hopf algebra on $ \sC(G,A)$, in the sense of the previous monoidal categories.
\end{prop}
 \medskip
 The following result describes the structure of the Hopf algebras of functions 
 $$\Q_p/p^{r+1}\Z_p \to \Z_p/p^{a+1}\Z_p
 \;,
 $$
  for any $r,a \in \Z$ and $a \geq 0$ in terms of the functions $x_i$ 
 $$ x_i :  \Q_p/ p^{i+1}\Z_p \longrightarrow \F_p 
 $$
 introduced in \eqref{wittcomp3}. See also Remark~\ref{Barsper}. 
 
 \begin{prop} \label{tesiunif} For any $i \in \Z$, let  $x_i$ be as in \eqref{wittcomp3} and let $X_i$ be indeterminates.
For $r \in \Z$ and $i \in \Z_{\geq 0}$,   let $\F_p(r,i)$ denote the $\F_p$-algebra  
$$
 \F_p[X_r,X_{r-1},X_{r-2},\dots,X_{r-i}]/(1-X_{r}^{p-1}, 1-X_{r-1}^{p-1},\dots, 1-X_{r-i}^{p-1}) \;.
 $$
 The dimension of $\F_p(r,i)$  as a $\F_p$-vector space
is $(p-1)^{i+1}$. 
 Let $X_{r,i}:= (X_{r-i}, X_{r-i+1},\dots,X_{r-1},X_r)$ be viewed as a Witt vector of length $i+1$ with coefficients in $\F_p(r,i)$.
 We make $\F_p(r,i)$ into an $\F_p$-Hopf algebra by setting
 $$
 \P X_{r,i} = X_{r,i} \otimes_{\F_p} 1 + 1 \otimes_{\F_p} X_{r,i} \;.
 $$
For any $i=0,1,\dots$, the map $\F_p$-algebra map $\F_p(r,i+1) \to \F_p(r,i)$ sending $X_{r-j}$ to $X_{r-j}$ if $0 \leq j \leq i$ and $X_{r-i-1}$ to $0$ 
 is  a homomorphism of $\F_p$-Hopf algebras. 
  Then, in the category $\LMu_{\F_p}$
  \ben 
  \item The map
\beq \begin{split}
\F_p(r,i) &\longrightarrow  \sC(p^{r-i}\Z_p/p^{r+1}\Z_p,\F_p)\\
X_j&\longmapsto x_j\;\;,\;\;\mbox{for}\;\; r-i \leq j \leq r\;,
\end{split}
\eeq
is an isomorphism of $\F_p$-Hopf algebras.
   \item
the $\F_p$-Hopf algebra $\sC(\Q_p/p^{r+1}\Z_p,\F_p)$ equals 
$$
\F_p(r,\infty) := \limPROu_{i \to +\infty} \F_p(r,i) 
$$
with the prodiscrete topology;
 \item 
 the topological $\F_p$-algebra $\sC_\unif(\Q_p/p^{r+1}\Z_p,\F_p)$ equals $\F_p(r,\infty)$
equipped with the discrete topology.
\een
 \end{prop}
 \begin{proof} Parts 1 and 2 are \cite[Teorema 3.31]{MA34}.
 Part 3 follows by forgetting the topology.  \end{proof}
 \begin{rmk} \label{perfalg}
 Notice that the $\F_p$-algebras $\F_p(r,i)$ are perfect. 
 \end{rmk}
  \begin{cor} \label{tesiunif2} For $r \in \Z$ and $i,a\in \Z_{\geq 0}$
   \ben 
 \item 
 the topological $\Z_p/p^{a+1}\Z_p$-algebra $\sC_\unif(\Q_p/p^{r+1}\Z_p,\Z_p/p^{a+1}\Z_p)$ equals
\beq \label{tesiform1}
{\rm W}_{a}(\sC_\unif(\Q_p/p^{r+1}\Z_p,\F_p)) =   {\rm W}_{a}(\F_p(r,\infty))
\eeq
equipped with the discrete topology. Therefore,
\beq \label{tesiform11}
\sC_\unif(\Q_p/p^{r+1}\Z_p,\Z_p) = {\rm W}(\sC_\unif(\Q_p/p^{r+1}\Z_p,\F_p)) =   {\rm W}(\F_p(r,\infty))
\eeq
equipped with the $p$-adic topology.
 \item
the $\Z_p/p^{a+1}\Z_p$-Hopf algebra $\sC(\Q_p/p^{r+1}\Z_p,\Z_p/p^{a+1}\Z_p)$ equals 
\beq \label{tesiform2}
{\rm W}_{a}(\sC(\Q_p/p^{r+1}\Z_p,\F_p)) =  {\rm W}_{a}(\F_p(r,\infty))
\eeq
with the prodiscrete topology. Therefore,
\beq \label{tesiform21}
\sC(\Q_p/p^{r+1}\Z_p,\Z_p) = {\rm W}(\sC(\Q_p/p^{r+1}\Z_p,\F_p)) =   {\rm W}(\F_p(r,\infty))
\eeq
equipped with the product topology of the prodiscrete topology of $\F_p(r,\infty)$ on the components.
\een
  \end{cor} 
  \begin{defn} \label{unif(r,a)} 
  We set 
  $$\sC = \sC(\Q_p,\Z_p) \;.
  $$ 
  For any $r,a \in \Z$ with $a  \geq 0$, we define a  Fr\'echet $\Z_p$-subalgebra of $\sC$
  $$
  \sC_{r,a} := \{f \in \sC  \,|\, f(x+p^{r+1}\Z_p) \subset f(x) +p^{a+1}\Z_p\;\, , \,\forall \,x\in \Q_p\,\}\;.
  $$
  Let $F$ be the set-theoretic map
 \beq \label{frobenius} \begin{split}
 F:\sC &\longrightarrow \sC \\
 f &\longmapsto f^p  
 \end{split} 
\eeq
  \end{defn}
Then
 \beq \sC_{r,a+1} \subset \sC_{r,a} \;\;\mbox{and}\;\; \sC_{r,a}  \subset \sC_{r+1,a}  
 \eeq
\beq \label{unifa1} 
  p^{a+1}\sC   \subset   \sC_{r,a}  
\eeq
  is an  ideal of $\sC_{r,a}$, and $F$ induces a map 
\beq
F : \sC_{r,a} \longrightarrow \sC_{r,a+1} \;.
\eeq
  There exists a 
  canonical map 
  \beq
  \label{unif(r,a)proj1}
  \begin{split}
  R_{r,a}:  \sC_{r,a} & \longrightarrow \sC(\Q_p/p^{r+1}\Z_p,\Z_p/p^{a+1}\Z_p )  \\
   f &\longmapsto R_{r,a}(f)
   \end{split}
   \eeq
   such that  
   $$\pi_{a+1} \circ f =  R_{r,a}(f) \circ \pi_{r+1}$$
   which sits in the exact sequence 
\beq
  \label{unif(r,a)ex}
   0 \longrightarrow p^{a+1} \sC \longrightarrow \sC_{r,a}  \map{R_{r,a}} \sC(\Q_p/p^{r+1}\Z_p,\Z_p/p^{a+1}\Z_p ) 
   = {\rm W}_{a}(\F_p(r,\infty)) \longrightarrow 0
\eeq 
  We conclude
  \begin{prop}\label{itext} For any $r \in \Z$ and any $a \in \Z_{\geq 1}$, the map $f \mapsto \pi_1 \circ f$ induces an isomorphism
  \beq \label{itext1} 
  \sC_{r,a} / p \sC_{r,a-1} \iso \sC_\unif(\Q_p/p^{r+1}\Z_p,\F_p) \;.
\eeq
For $a=0$ we similarly have 
  \beq \label{itext2} 
   \sC_{r,0} / p \sC  \iso \sC_\unif(\Q_p/p^{r+1}\Z_p,\F_p) \;.
  \eeq
The inverse of the  isomorphism of discrete $\F_p$-algebras
  \beq \label{itext3} 
  \sC_{r,a} / p \sC_{r,a-1} \iso \sC_{r,0} / p \sC 
\eeq
is provided by the map 
\beq \label{itext4} \begin{split}
F^a: \sC_{r,0} / p \sC &\iso   \sC_{r,a} / p \sC_{r,a-1} \\
f & \longmapsto f^{p^a} \;.
\end{split}
\eeq
 \end{prop}
  \begin{proof} The first formula follows from \eqref{unifa1} and \eqref{unif(r,a)ex}. In fact, 
  $$  \sC_{r,a} / p \sC_{r,a-1} = (\sC_{r,a}/p^a\sC) / p (\sC_{r,a-1} /p^{a-1}\sC) = {\rm W}_{a}(\F_p(r,\infty))/p{\rm W}_{a-1}(\F_p(r,\infty)) = \F_p(r,\infty)\;.
  $$
  Similarly for the
   other formulas.
   \end{proof}
  By iteration, we get
  \begin{cor} \label{itextcor}
    \beq \label{itext0} 
  \sC_{r,a} / p^{a+1} \sC \iso \sC(\Q_p/p^{r+1}\Z_p,\Z_p/p^{a+1}\Z_p) = {\rm W}_{a}(\F_p(r,\infty)) \;.
\eeq
For any $f \in  \sC_{r,a}$ there exist $f_0,f_1,\dots,f_a \in \sC_{r,0}$, well determined modulo $p\sC$, such that
  \beq \label{itext01} 
 f  \equiv  f_0^{p^a} +  p  f_1^{p^{a-1}} + p^2  f_2^{p^{a-2}}  + \dots + p^a  f_a \;\;\;{\rm mod}\; p^{a+1} \sC  \;.
\eeq 
  \end{cor}
 \end{section}
 \begin{section}{$p$-adically entire functions bounded on $\Q_p$} \label{entfcts}  

   We prove here the statements announced in the Introduction, namely  Theorem~\ref{bohrpadthm}, Theorem~\ref{bohrpadthm2},
Proposition~\ref{2struct},   Proposition~\ref{hopfpad}, Proposition~\ref{fejer21}, Proposition~\ref{fejer22}, 
  and Theorem~\ref{bohrpadcor}.  We assume $q=p$ from now on, so in particular $\Psi$ stands for $\Psi_p$. 
  \endgraf \smallskip
  We start with  the proof of Theorem~\ref{bohrpadthm}.
  \begin{proof} (of Theorem~\ref{bohrpadthm}) It suffices to prove the statement over $\Z_p$. Notice that 
$$\Z_p[\Psi(\lambda x)\,|\, \lambda \in \Q_p^\times] = \Z_p[\Psi(\lambda p^{-i} x)\,|\, i\in \Z \,,\,\lambda \in \Z_p^\times] \;.
$$
Both rings 
$\Z_p[[(\lambda x)_i]\,|\, i \in \Z\,,\, \lambda  \in \Z_p^\times]$ and 
$ \Z_p[\Psi(\lambda p^{-i} x)\,|\, i \in \Z \,,\,\lambda \in \Z_p^\times]$ are contained in the $\Z_p$-Banach ring $\sC_\unif(\Q_p,\Z_p)$ 
which may be identified with ${\rm W}(\sC_\unif(\Q_p,\F_p))$ equipped with the $p$-adic topology. 
Then $AP_{\Z_p}$ consists of ${\rm W}( \F_p[(\lambda x)_i\,|\, i \in \Z\,,\,  \lambda  \in \Z_p^\times])$. Notice that $\F_p[(\lambda x)_i \,|\, i \in \Z\,,\,  \lambda  \in \Z_p^\times]$ is a perfect subring of the perfect ring $\sC_\unif(\Q_p,\F_p)$, since $(\lambda x)_i^p=(\lambda x)_i$, for any $i,\lambda$. 
It suffices to prove
\begin{lemma}  \label{psicomplambda} For any fixed $\lambda \in \Z_p^\times$, the closure  of 
$\Z_p[\Psi(\lambda p^ix)\,|\, i =0,1,2,\dots ]$ in ${\rm W}(\sC_\unif(\Q_p,\F_p))$ coincides with 
of ${\rm W}( \F_p[(\lambda x)_i\,|\, i = 0, -1,-2,\dots])$. 
\end{lemma}
\begin{proof}
We may as well assume $\lambda=1$ and prove
\begin{sublemma} \label{psicomp} The closure  of 
$\Z_p[\Psi(p^ix)\,|\, i =0,1,2,\dots ]$ in  ${\rm W}(\sC_\unif(\Q_p,\F_p))$  coincides with 
of ${\rm W}( \F_p[x_i\,|\, i = 0, -1,-2,\dots])$.
\end{sublemma}
\begin{proof} Let $C$ be the closure  of 
$\Z_p[\Psi(p^ix)\,|\, i =0,1,2,\dots ]$ in  ${\rm W}(\sC_\unif(\Q_p,\F_p))$. The formula
$$
[x_{-i}] = \lim_{N \to \infty} \Psi(p^ix)^N 
$$
shows that ${\rm W}(\sC_\unif(\Q_p,\F_p)) \subset C$. It will suffice to show that, as functions $\Q_p \to \Z_p$ 
$$
\Psi(x) \in {\rm W}( \F_p[x_i\,|\, i = 0, -1,-2,\dots]) \;.
$$
We write the restriction of $\Psi(x)$ to a function $\Q_p \to \Z_p$ as
$$
\Psi(x) = (\Psi_0(x),\Psi_1(x),\Psi_2(x), \dots) 
$$
with $\Psi_i \in \sC_\unif(\Q_p,\F_p)$ and $\Psi_0(x)=x_0$. 
We have, from \eqref{functeq}, the formula in $\sC_\unif(\Q_p,\Q_p)$
\beq\label{functeqQ1}\Psi(x) + p^{-1}\Psi(px)^p + \dots + p^{-i}\Psi(p^ix)^{p^i} +\dots  = x= (\dots, x_{-i},\dots,x_{-2},x_{-1};x_0,\ast,\ast,\dots)
\eeq
From  \eqref{functeqQ1} we deduce that, as functions in $\sC_\unif(\Q_p,p^{-i}\Z_p)$
\beq\label{functeqQ2}\Psi(x) + p^{-1}\Psi(px)^p + p^{-2}\Psi(p^2x)^{p^2} + \dots + p^{-i}\Psi(p^ix)^{p^i}=  (x_{-i},\dots,x_{-2},x_{-1};x_0,\ast,\ast,\dots)
\eeq
This shows, inductively on $i$, that  
$$\Psi_i \in \F_p[x_j\,|\, j = 0, -1,-2,\dots,-i] \;.$$
\end{proof}
\end{proof}
\end{proof}

  \medskip
  \begin{defn} \label{periodrdef} Let $r \in \Z$ and $a \in \Z_{\geq 1}$.  
 We define $\sE_{r,a}^\circ$ (resp. $\sT_{r,a}^\circ$) to be  the $\Z_p$-subalgebra  of $\sE^\circ_{p^r}$ (resp. of 
 $\sT_{p^r}^\circ$)  (\cf Definition~\ref{Elambdadef})
  consisting of those functions $f$ such that
  \beq \label{periodr}
 f(x+p^{r+j}\C^\circ_p) \subset f(x) + p^{a+j}  \C^\circ_p\;\;,\;\;  \forall \;x \in \Q_p \;\;\mbox{and} \;\;  \forall \;j \in \Z_{\geq 1}\;. 
 \eeq  
  \end{defn}
  \begin{rmk}
 For the rest of this section the statements valid for the rings $\sE_{r,a}^\circ \subset \Q_p\{x\}$ hold equally well, and with the same proof,  for the rings $\sT_{r,a}^\circ \subset \cO(\Sigma_{p^{-r}})^\circ$. For short, we deal with the former only. 
\end{rmk}
Notice that
  \beq \label{periodrdef10} \sE_{r,a+1}^\circ  \subset \sE_{r,a}^\circ  \subset \sE_{r+1,a+1}^\circ 
  \;\;\mbox{and}\;\; p  \sE_{r,a}^\circ  \subset \sE_{r,a+1}^\circ  
  \eeq 
  and that we have a map $F$  as in Definition~\ref{unif(r,a)} such that
    \beq \label{periodrdef11} F(\sE_{r,a}^\circ)  \subset \sE_{r,a+1}^\circ   
   \;.
  \eeq 

  \begin{rmk} \label{changenot} 
    We have 
    $$\sE_{p^r}^\circ := \sE_{r,0}^\circ \;.
    $$
    We already proved (Proposition~\ref{unifPsi} and Corollary~\ref{basicest0}) that $\Psi(x) \in \sE_{0,0}^\circ$. Therefore,
 for any  $i \in \Z_{\geq 0}$ and $\ell =0,1,\dots,p-1$, the function $\Psi(p^{i-r} x)^{\ell p^a}$ belongs to  $\sE_{r-i,a}^\circ \subset \sE_{r,a}^\circ$. 
\end{rmk}  
\begin{lemma}\label{standardvsunif}
If a sequence of functions $n \mapsto f_n \in \sE_{r,a}^\circ$ (resp. $\in \sT_{r,a}^\circ$) converges  to $f \in \C_p\{x\}$ (resp. to $f \in \cO(\Sigma_{p^{-r}})^\circ$)
uniformly on 
bounded subsets of $\C_p$  (resp. of $\Sigma_{p^{-r}}$)
then $f \in  \sE_{r,a}^\circ$ (resp. $\in \sT_{r,a}^\circ$). Therefore $\sE_{r,a}^\circ$ (resp. $\sT_{r,a}^\circ$) is a closed $\Z_p$-subalgebra of $\C_p\{x\}$ (resp. of $(\cO(\Sigma_{p^{-r}})^\circ,{\rm standard})$). The induced Fr\'echet algebra structure 
  on $\sE_{r,a}^\circ$ (resp. on $\sT_{r,a}^\circ$) will be called \emph{standard}. 
\end{lemma}
 \begin{proof} We deal, to fix ideas, with the case of $\sE_{r,a}^\circ$. We show that for any $c \in \Q_p$ and $j=0,1,\dots$, 
 $$
 f(c + p^{r+j +1} \C_p^\circ) \subset f(c) + p^{a+j+1} \C_p^\circ \;.
 $$
 By assumption, for any $s,t \in \N$, there exists $N=N_{s,t}$ such that if $n \geq N$, then
 $$
 (f_n-f)(p^{-s}\C_p^\circ) \subset p^t \C_p^\circ \;.
 $$ 
 So, for $c$ and $j$ as before, let $s$ be such $c + p^{r+j+1}\C_p^\circ \subset p^{-s} \C_p^\circ$, and let $t$ be $\geq j+a+1$.
 Then, for any $n \geq N_{s,t}$, 
 $$
 (f_n-f)(c + p^{r+j+1} \C_p^\circ) \subset  (f_n-f)(p^{-s}\C_p^\circ) \subset p^t \C_p^\circ  \subset p^{j+a+1} \C_p^\circ \;.
 $$
 Therefore $f \in  \sE_{r,a}^\circ$. 
  \end{proof}
  Notice that Proposition~\ref{2struct} follows from Lemma~\ref{standardvsunif}, by taking $a=0$. 
  \par
 Let $r,a$ be as in Definition~\ref{periodrdef}. 
 Any function $f \in \sE_{r,a}^\circ$
 induces a continuous function $f_{|\Q_p}: \Q_p \to \Z_p$.
 The 
 $\Z_p$-linear map 
\beq \label{periodrrast} 
Res^\circ: (\sE_{r,a}^\circ,{\rm standard}) \longrightarrow \sC_{r,a} \subset \sC(\Q_p,\Z_p)\;\;,\;\; f \longmapsto f_{|\Q_p}\;,
\eeq
is  continuous and injective.  
By composition, we obtain, for any $r \in \Z$ and any $a,h=0,1,\dots$,  a morphism  
\beq\label{restrholst}
 R_{r,a} \circ Res^\circ : (\sE_{r,a}^\circ,{\rm standard}) \longrightarrow \sC(p^{r-h}\Z_p/p^{r+1}\Z_p,\Z_p/p^{a+1}\Z_p) = {\rm W}_{a}(\F_p(r,h))\;, 
\eeq
where the r.h.s. is equipped with the  topology of \eqref{tesiform2}.
The kernel  of that  map   is the set of $g \in \sE_{r,a}^\circ$ such that $-\log ||g||_{p^{h-r}}   \geq a+1$. 
From \eqref{periodrrast} we  also get   maps of Fr\'echet $\Z_p$-algebras
\beq \label{periodrra2}
Res^\circ:  (\sE^\circ_\lambda, \{||~||_{p^r\Z_p}\}_{r \in \Z})^\wedge \longrightarrow \sC(\Q_p,\Z_p)\;\;,\;\; f \longmapsto f_{|\Q_p} \;,
\eeq
\beq \label{periodrra3}
Res^\circ:  (\sT^\circ_\lambda, \{||~||_{p^r\Z_p}\}_{r \in \Z})^\wedge \longrightarrow \sC(\Q_p,\Z_p)\;\;,\;\; f \longmapsto f_{|\Q_p} \;.
\eeq
\begin{lemma} \label{periodrexpWittintro} Let $r \in \Z$ and $a \in \Z_{\geq 0}$ be as before. 
 \hfill
 \ben 
 \item  
 Any series of functions of the form
  \beq \label{periodrexpWitt1}
 \sum_{\ell=0}^{p-1} \sum_{i =0}^\infty c_{\ell,a,i} \Psi(p^{i-r} x)^{\ell p^a}\;\;,\;\; c_{\ell, a,i}  \in \Z_p \;,
\eeq
converges in the standard Fr\'echet topology of $\Q_p\{x\}$ to an element of $\sE_{r,a}^\circ$   along the filter of cofinite subsets of $\{0,1,\dots,p-1\} \times \Z_{\geq 0}$.
\item   For any element $f \in  \sC_{r,a}$ and for any $s =0,1,2,\dots$  there exist 
uniquely determined 
elements $c_{\ell,b,i} = c_{\ell,b,i}^p \in \Z_p$,  such that 
for 
$$f_{r,a} :=  \sum_{b=0}^{a} \sum_{\ell=0}^{p-1} \sum_{i =0}^\infty c_{\ell,b,i} p^b \Psi(p^{i-r} x)^{\ell p^{a-b}} 
\in \sE_{r,a}^\circ \;,$$
where the infinite sum converges in the standard Fr\'echet topology of $\sE_{r,a}^\circ$, we have
\beq \label{wittfourier}  
-\log ||(f - f_{r,a})||_{p^{-r}\Z_p} \geq a+1  \;.
\eeq
Same statement for $\sE_{r,a}^\circ$ replaced by  $\sT_{r,a}^\circ$.
\item   For any element $f \in  \sC_{r,a}$ and any $h=0,1,\dots$,   there exist 
uniquely determined 
elements $c_{\ell,b,i} = c_{\ell,b,i}^p \in \Z_p$,  such that 
for 
$$f_{r,a,h} :=  \sum_{b=0}^{a} \sum_{\ell=0}^{p-1} \sum_{i =0}^h c_{\ell,b,i} p^b \Psi(p^{i-r} x)^{\ell p^{a-b}} 
  \;,$$ 
\beq \label{wittfourier1} -\log ||(f - f_{r,a,h})||_{p^{h-r}\Z_p} \geq a+1  
  \;.
\eeq
\item The map \eqref{restrholst} is surjective. 
\item The maps \eqref{periodrra2} and \eqref{periodrra3} are the  isomorphisms of Theorem~\ref{bohrpadthm2}.
 \een
 \end{lemma}
 \begin{proof} The first part is clear. As for the second, 
 we observe that, for any $b =0,1,\dots,a$, the map $R_{r,a} \circ Res^\circ$  
  transforms the function $p^{b} \Psi(p^{i-r} x)^{\ell p^{a-b}}$, for $\ell =0,1,\dots,p-1$,  into the Witt vector 
 $$(0,\dots,0,w_b= x_{r-i}^\ell,0,\dots,0) \in {\rm W}_{a}(\F_p(r,\infty))\;,
 $$
 where $x_{r-i}^\ell$ is placed at the $b$-th level.
 Since any $y \in \F_p(r,\infty)$ admits a unique expression as a  sum, convergent in the 
 prodiscrete topology of  $\F_p(r,\infty)$, 
 $$
 y = \sum_{\ell=0}^{p-1} \sum_{i =0}^\infty \gamma_{\ell,i}  x_{r-i}^\ell \;\;,\;\; \gamma_{\ell,i}  \in \F_p\;,
 $$
  is clear that any $w = (w_0,w_1,\dots,w_a) \in  {\rm W}_{a}(\F_p(r,\infty))$ admits a unique expression 
  as a sum
  $$
   \sum_{b=0}^{a} \sum_{\ell=0}^{p-1} \sum_{i =0}^\infty [\gamma_{\ell,b,i}] (0,\dots,0,w_b=x_{r-i}^\ell,0,\dots,0) 
   $$
   which in turn converges in the  prodiscrete topology of ${\rm W}_{a}(\F_p(r,\infty))$. More precisely, for any $a,h =0,1,\dots$,  we can determine coefficients $c_{\ell,b,i} \in \Z_p$ such that 
   $$
  w -  \sum_{b=0}^{a} \sum_{\ell=0}^{p-1} \sum_{i =0}^{h} c_{\ell,b,i}  x_{r-i}^\ell
   $$
   has zero image in ${\rm W}_{a}(\F_p(r,h))$. So, the function
   $$f_{r,a,h} :=  \sum_{b=0}^{a} \sum_{\ell=0}^{p-1} \sum_{i =0}^{h} c_{\ell,b,i} p^b \Psi(p^{i-r} x)^{\ell p^{a-b}} 
 \;,$$
   is such that 
   $$\min \{v_p(f_{r,a,h}(x) - f(x)) \,|\, x \in p^{r-h} \Z_p + p^r \C_p^\circ\} \geq a+1 \;.
   $$
  Finally,  we already observed that the kernel of the map
  \eqref{restrholst} consists of the elements 
$g \in \sE_{r,a}^\circ$ such that $-\log ||g||_{p^{h-r}\Z_p}  \geq a+1$. 
This proves $2$, $3$ and $4$. 
\par
As for the last part of the statement, we pick any $f \in \sC_\unif(\Q_p,\Z_p)$ and a natural number $N=0,1,\dots$. 
Then there exists an $M =0,1,\dots$ and $f_M \in \sC(\Q_p/p^M\Z_p, \Z_p)$ such that $w_\infty(f-f_M) \geq N$.       It will suffice to determine an element   $g \in \Z_p[ \Psi(\lambda x)\,|\,\lambda \in \Q_p^\times]$ such that $w_\infty (g-f_M) \geq N$. We then pick $r \in \Z$ and $a \in \Z_{\geq 0}$ so that $r+1 \geq M$ and $a+1 \geq N$.   
The statement follows from the surjectivity of  \eqref{restrholst}.
This concludes the proof.
 \end{proof} 
 As a corollary, we obtain the   proof of Propositions~\ref{fejer21} and \ref{fejer22}.
 We now give the proof of Proposition~\ref{hopfpad}.
 \begin{proof} (of Proposition~\ref{hopfpad}) We discuss $(\sE^\circ_\lambda, {\rm standard})$ in order to fix ideas. The case of  $(\sT^\circ_\lambda, {\rm standard})$ is analogous. 
The coproduct of $\sE^\circ_\lambda$ originates from \eqref{covsum2}
\beq \label{covsum20}\begin{split}
x \mapsto \Psi(x \wt_{\Z_p} 1& + 1 \wt_{\Z_p} x) =  \Phi (\Psi(x\wt_{\Z_p} 1), \Psi(px \wt_{\Z_p} 1),\dots;\Psi(1 \wt_{\Z_p} x),\Psi(1 \wt_{\Z_p} px),\dots) =\\
&\Phi (\Psi(x) \wt_{\Z_p} 1, \Psi(px) \wt_{\Z_p} 1 ,\dots;1 \wt_{\Z_p} \Psi(x),\wt_{\Z_p} \Psi(1 px),\dots)
\end{split}
\eeq
 and the identification \eqref{ident1}. The fact that $\sE^\circ_\lambda$ only depends upon $|\lambda|$ follows from the fact that, for any $f \in  \C\{x\}$,  the map $\Q_p \to \C\{x\}$, $a \mapsto f(a x)$ is continuous. 
 For any $n \in \Z$, the map $n \iota: \Psi(\lambda^{-1}p^jx) \mapsto \Psi(\lambda^{-1}p^j n x)$, for any $j=0,1,\dots$, is an endomorphism of 
 $\sE^\circ_\lambda$. By continuity, we obtain a map $a \iota : \sE^\circ_\lambda \to \sE^\circ_\lambda$, for any $a \in \Z_p$. 
If $m,n \in \Z$ are such that $mn = 1 + a p^N$, for $a \in \Z$ and $N \in \Z$, $N >>0$, $\Psi(\lambda^{-1}p^j mn x)$ is close to $\Psi(\lambda^{-1}p^j x)$. Again by continuity we find that if $a\in \Z_p^\times$,   $a \iota$ is an automorphism of $\sE^\circ_\lambda$. 
\end{proof}
We finally prove our Uniform Approximation Theorem~\ref{bohrpadcor}.
\begin{proof} We discuss the integral case only; the bounded case follows directly.  We first observe that a $\LMu_{\Z_p}$-morphism 
$$
( \APH_{0,\Z_p},{\rm strip})  = \limINDU_{\rho \to 0} ( \APH_{\Z_p}(\Sigma_\rho), {\rm strip}) \longrightarrow (AP_{\Z_p},w_\infty)
$$
exists because so does, for any $\rho >0$, the  morphism  $(\APH_{\Z_p}(\Sigma_\rho), {\rm strip}) \to (AP_{\Z_p},w_\infty)$. Moreover, that morphism is injective. An element of $\APH_{0,\Z_p}$ is represented by a sequence 
$P_{\rho_n} \in \Z_p[\Psi(x/\lambda)\,|\, \lambda \in \Q_p^\times \,]$  with $\rho_n$ decreasing to $0$, such that for any $\veps >0$, there exists $N_\veps$ such that for any $m \geq n \geq N_\veps$,
$$
||P_{\rho_n} - P_{\rho_m}||_{\Q_p,\rho_m} < \veps \;.
$$

Let $f \in AP_{\Z_p}$ and let $N \in \Z_{>0}$. By definition of u.a.p. functions, there exists a polynomial 
$$
P_N := \sum_{\lambda \in \Q_p^\times} a_\lambda \Psi(x/\lambda)
$$
where $a_\lambda \in \Z_p$ $=0$ for almost all $\lambda$, such that 
$$
w_\infty (f - P_N) > N \;.
$$
By \eqref{unifPsi11} of Theorem~\ref{psibddsuper}, for any $N >0$ there exists $\rho_N >0$ such that 
$v_p(P_N(a + x) - P_N(a)) > N$, for any  $a \in \Q_p$ and $x \in \C_p$, $|x| \leq \rho_N$. We may assume that the sequence $N \to \rho_N$ decreases to $0$. 
We deduce that for $M \geq N$  
$$
-\log ||P_N - P_M||_{\Q_p, \rho_M} > N \;.
$$
 So, the sequence $N \mapsto P_N$ represents a germ $P \in \APH_{0,\Z_p}$ whose restriction to $\Q_p$ is $f$. 
\end{proof}
\begin{rmk} \label{caution}
We are not asserting here that there should be a $p$-adic strip $\Sigma_\rho$ around $\Q_p$ on which $f$ extends analytically. 
In fact, an inductive limit in the category $\LMu_{\Z_p}$ is not necessarily supported by a set-theoretic inductive limit (see section~\ref{lincat} of Appendix A below) and similarly for a locally convex inductive limit of Banach spaces. 
\end{rmk}
\end{section}
\begin{section}{Appendix A. Non archimedean topological algebra} \label{Conventions}
A prime number $p$ is fixed throughout this paper and $q=p^f$ is a power of $p$. So, $\Q_q$ will denote the unramified extension of $\Q_p$ of degree $f$, and $\Z_q$ will be its ring of integers.      Unless otherwise specified, a  \emph{ring} is meant to be commutative with 1.  
\begin{subsection}{Linear topologies} \label{lincat}
Let $k$ be a  separated and complete linearly topologized ring; we will denote by $\cP(k)$ the family of open ideals of $k$. We will consider the category $\LMu_k$ of separated and complete linearly topologized  $k$-modules $M$ such that 
the map multiplication by scalars
$$
k \times M \longrightarrow M\;\;,\;\; (r,m) \longmapsto rm
$$
is \emph{uniformly continuous} for the product uniformity of $k \times M$. Morphisms of $\LMu_k$ are continuous $k$-linear maps.  This is the classical category of \cite[Chap. III, \S 2]{algebracomm}. See \cite{closed}  for more detail.  
\begin{rmk}\label{warning} 
All over this paper we will assume that in a topological ring $R$ (resp. topological $R$-module $M$), the product (resp. the scalar product) map $R \times R \to R$ (resp. $R \times M \to M$)
is at least continuous for the product topology of $R \times R$ (resp. of $R \times M$); morphisms will be continuous morphisms of rings (resp. of $R$-modules).
\par By a  \emph{non-archimedean} (\emph{n.a.}) 
ring  $R$ (resp.  $R$-module $M$) we mean a topological ring $R$ (resp.   $R$-module $M$)  equipped with a topology for which a basis of neighborhoods of $0$ consists of additive subgroups and   additive translations are  homeomorphisms.  So, any valued non-archimedean field $K$ is a n.a. ring in the previous sense and, if $K$ is non-trivially valued,  the category $\LC_K$ of locally convex $K$-vector spaces \cite{schneider} is a full subcategory of the category of n.a. $K$-modules. But, such a field $K$ is never a linearly topologized ring. The ring of integers  $K^\circ$ is indeed linearly topologized, but no  non-zero object of $\LC_K$ is an object of $\LMu_{K^\circ}$. 
\end{rmk}
\begin{defn}\label{strictly closed} Let $R$ be a topological ring and $M$ be a topological $R$-module. 
A closed topological $R$-submodule $N$ of $M$ is said to be \emph{strictly} closed if it is endowed with the subspace topology of $M$.
\end{defn}
For any object $M$ of $\LMu_k$, $\cP(M)$ will denote the family of open $k$-submodules of $M$. The category $\LMu_k$ admits  all limits and colimits. 
The former are calculated in the category of $k$-modules but not the latter. So, a limit will be denoted by $\limPROu$ while a colimit will carry an apex $(-)^u$ as in  $\limINDU$.
In particular, for any family $M_\alpha$, $\alpha \in A$, of objects of $\LMu_k$, the direct sum and direct product will be denoted by 
$$
\SUMU_{\alpha \in A} M_\alpha \;\;,\;\; \prod_{\alpha \in A} M_\alpha \;,
$$
respectively. We explicitly notice that $\SUMU_{\alpha \in A} M_\alpha$ is the completion of the algebraic direct sum 
$\bigoplus_{\alpha \in A} M_\alpha$ of the algebraic $k$-modules $M_\alpha$'s, equipped with the $k$-linear topology for which a fundamental system of open $k$-modules consists of the $k$-submodules 
$$
\bigoplus_{\alpha \in A} (U_\alpha + I M_\alpha)\;\mbox{such that} \;U_\alpha \in \cP(M_\alpha)\; \forall \alpha\;, \; \mbox{and}\;I \in \cP(k)\; \mbox{is independent of} \, \alpha \; .
$$
 Then the  $k$-module underlying $\SUMU_{\alpha \in A} M_\alpha$ in general properly contains the algebraic direct sum $\bigoplus_{\alpha \in A} M_\alpha$.
 It will also be useful to introduce the \emph{uniform box product} of the same family 
\beq
\label{square}
\PRODSCU{\alpha \in A} M_\alpha 
\eeq
which, set-theoretically, coincides with $\prod_{\alpha \in A} M_\alpha$ but whose family of open submodules consists 
of all $U := \prod_{\alpha \in A} U_\alpha$, with $U_\alpha \in \cP(M_\alpha)$, such that there exists $I_U \in \cP(k)$ such that 
$I_U M_\alpha \subset U_\alpha$, for any $\alpha \in A$. The category $\LMu_k$, equipped with the tensor product 
$\wt^u_k$ of \cite[{$\bf 0$}.7.7]{EGA} (see also \cite[Chap. III, \S 2, Exer. 28]{algebracomm}) is a symmetric monoidal category.  The category of monoids of $\LMu_k$ is denoted by $\ACLMu_k$.  
\par
For two objects $M$ and $N$ of $\LMu_k$, we have 
\beq \label{tensproj}
M \wt^u_k N = \limPROu_{P \in \cP(M),Q \in \cP(N)} M/P \otimes_k N/Q  
\eeq
so that $\wt^u_k$ commutes with filtered projective limits in $\LMu_k$.  
\end{subsection}
\begin{subsection}{Semivaluations} \label{semival}
We describe here full subcategories of $\LMu_k$, and special base rings $k$, of most common use. 
  We denote by $\Z_{(p)} = \Q \cap \Z_p$, the localization of $\Z$ at $(p)$.
Then $\C_p$ will be the completion of a fixed algebraic closure of $\Q_p$.  
On  $\C_p$  we use the absolute value  $|x|=|x|_p = p^{-v_p(x)}$, for the $p$-adic valuation $v=v_p$, with $v_p(p)=1$, and $x \in \C_p$.   
\endgraf
\begin{defn} \label{semival1}
A \emph{semivaluation}  on a ring $R$ is a map $w:R \to \R  \cup \{+\infty\}$ such that $w(0) = +\infty$, 
$w(x+y) \geq \min(w(x),w(y))$ and $w(xy) \geq w(x) + w(y)$, for any $x,y \in R$. We will say that a semivaluation is \emph{positive} if it takes its values in $\R_{\geq 0}  \cup \{+\infty\}$. 
\end{defn}
\begin{rmk} \ben
\item If $w_1,w_2,\dots,w_n$ are a finite set of semivaluations  on the ring $R$, so is their infimum
$$
w := \inf_{i=1,\dots,n} w_i \;.
$$
\item
The \emph{trivial valuation} $v_0:R \to \{0,+\infty\}$, which exists on any ring $R$, is (in our sense!) a positive semivaluation. 
\item We will indifferently use the multiplicative notation $|x|_w = \exp(-w(x))$.
\een
\end{rmk}
For any semivaluation $w$ of a ring $R$, the  family of 
\beq
 \label{semival11} 
 R_{w,c} :=   \{x \in R \,|\, w(x) \geq c\,\} 
\eeq
for ${c \in \R}$ is a fundamental set 
of open subgroups for a group topology of $R$. Moreover, $R_{w,0}$ is a subring of $R$ and all $R_{w,c}$ are $R_{w,0}$-submodules of $R$.  
A (\emph{multi}-) \emph{semivalued} ring $(R,\{w_\alpha\}_{\alpha \in A})$ is a ring $R$ equipped with a family $\{w_\alpha\}_{\alpha \in A}$ of semivaluations. A semivalued ring is endowed with the topology in which any $x \in R$ has a fundamental system of 
neighborhoods  consisting 
of the subsets 
$$
x + \bigcap_{\alpha \in F} R_{\alpha,c_\alpha} 
$$
where $F$ varies among finite subsets of $A$ and, for any $\alpha \in F$, $c_\alpha \in \R$. 
A \emph{Fr\'echet ring} (resp. \emph{Banach ring})  is a ring $R$ which is separated and complete in the topology  induced by a countable family   of semivaluations (resp. by a single semivaluation). If the semivaluations 
$w_\alpha$ are all positive, the Fr\'echet (resp. Banach)  ring $(R,\{w_\alpha\}_{\alpha \in A})$ is 
linearly topologized. We will call it a \emph{linearly topologized Fr\'echet} (resp. \emph{Banach}) \emph{ring}. 
When $R$ is an algebra over a Banach ring $(S,v)$, and the semivaluations $w_\alpha$ satisfy 
$$w_\alpha (xy) = v(x) + w_\alpha (y)\;\;\forall \;x \in S\;,\; y \in R\;,
$$
we also say that $R = (R, \{w_\alpha\}_{\alpha \in A})$ is a \emph{Fr\'echet} (resp. \emph{Banach}) \emph{$S$-algebra}.
In the particular case when $(S,v)$ is a complete non-trivially valued real-valued field $(K,v)$ a Fr\'echet or Banach $S$-algebra  is a Fr\'echet or Banach algebra over $K$ in the classical sense. Notice however that we allow $v$ to be the trivial valuation of $S$ or $K$. 
We denote by $\CLC_K$  the category of locally convex topological $K$-vector spaces of \cite{schneider}, where morphisms are continuous $K$-linear maps,  which are moreover separated and complete. 
\par 
We have the easy
 \begin{lemma}\label{tensorfrechet} Let  $(S,v)$ be a Banach ring and $(R, \{w_n\}_{n=1,2,\dots})$ be a Fr\'echet $S$-algebra. Let  $(R_n,w_n)$ be the separated completion of $R$ in the locally convex topology induced by the semivaluation  
$w_n$.  Assume $w_n(r) \geq w_m(r)$ for any $r \in R$ if $n \leq m$. Then, the identity of $R$ extends to a morphism 
$R_m \to R_n$ of Banach $S$-algebras and $R$ is the limit, in the category of n.a.  $S$-algebras, of the filtered projective system $(R_n)_n$. \par
In particular, 
 a $S$-subalgebra $T$ of $R$  is dense in $R$ if and only if it is dense in $R_n$, for any $n$.
  \end{lemma}
\end{subsection}
\begin{subsection}{Tensor products} \label{tensprod}
Let $(S,v)$ be a complete real-valued ring and let $R = (R, \{w_\alpha\}_{\alpha \in A})$ and $R' = (R', \{w'_\beta\}_{\beta \in B})$ be two Fr\'echet $S$-algebras. Then we define a Fr\'echet $S$-algebra $R\wt_{\pi,S}R'$ as the completion of the $S$-algebra $R \otimes_S R'$ in the topology induced by the following semivaluations  \cite[2.1.7]{BGR}, for any $(\alpha,\beta) \in A \times B$,
$$
w_{\alpha,\beta} (g) = \sup \l(\min_{1\leq i \leq n} w_\alpha(x_i) + w'_{\beta}(y_i)\r) \;,
$$
where the supremum runs over all possible representations
$$
g = \sum_{i=1}^n x_i \otimes y_i \;\;,\;\;x_i \in R\;,\;y_i \in R'\;.
$$
The following proposition follows immediately from Lemma~\ref{tensorfrechet}.
\begin{prop}\label{tensorfrechet2} Let  $(S,v)$ be a Banach ring and $(R, \{w_n\}_{n=1,2,\dots})$, 
$(R', \{w'_n\}_{n=1,2,\dots})$
 be two Fr\'echet $S$-algebra satisfying the assumption of Lemma~\ref{tensorfrechet}. 
 Then, with the same notation, $R\wt_{\pi,S}R'$ is the limit, in the category of n.a.  $S$-algebras, of the filtered projective system 
 of Banach $S$-algebras 
 $(R_n\wt_{\pi,S}R'_n)_n$.
 \end{prop}

Notice that
\ben
\item
 if $R$ and $R'$ are Fr\'echet algebras over a complete real-valued field $(K,v)$, with non-trivial valuation $v$, 
$R\wt_{\pi,K}R'$ coincides with  both the completed projective and the inductive tensor product of \cite{schneider} (\cf Lemma 17.2 and Lemma 17.6 of \lc);
\item  if $R$ and $R'$ are linearly topologized  Fr\'echet algebras over a linearly topologized Banach ring $(S,v)$, 
$R\wt_{\pi,S}R'$ coincides with $R \wt^u_S R'$.
\een

\end{subsection}

\end{section}
\begin{section}{Appendix B. Classical theory of almost periodic functions} \label{Canalog}
The main character of this paper, our function $\Psi$, shows many analogies with the classical holomorphic almost periodic functions of Bohr, Bochner, and Besicovitch \cite{Bes}. In fact many of the  subtle function theoretic  
difficulties which appear in the $p$-adic setting are also  encountered in classical Harmonic Analysis. 
We feel that a short presentation of the   basics of the classical theory might be useful. See also the survey article \cite{cooke}.
\subsection{Fej\'er's Theorem} \label{fejerthm} Let $(\sC_\unif^\bd(\R,\R), ||~||_\R)$ be the Banach algebra of bounded uniformly continuous functions $\R \to \R$, equipped with the supnorm on $\R$.
For $\lambda \in \R_{>0}$ let $\sP_{\R,\lambda} \subset \sC_\unif^\bd(\R,\R)$ be the strictly closed 
Banach subalgebra of continuous functions periodic of period $\lambda$.
 \par
Let us recall the classical \emph{Fej\'er's Theorem} \cite[\S 13.31]{Tit}. Let $f \in \sP_{\R,\lambda}$. The \emph{Fourier expansion of $f$} is 
 the formal  trigonometric series 
 $$
 \frac{a_0}{2} + \sum_{n=1}^\infty a_n \cos(\frac{2 \pi n}{\lambda} z) + b_n \sin(\frac{2 \pi n}{\lambda} z) \;,  
 $$
 with  
 $$
 a_n = \frac{2}{\lambda} \int_0^\lambda f(t) \cos(\frac{2 \pi n}{\lambda} z)dt \;\;,\;\; b_n = \frac{2}{\lambda}  \int_0^\lambda f(t) 
 \sin(\frac{2 \pi n}{\lambda} z) dt \;.
 $$
 The sequence of the partial sums 
 $$
 S_N (f) =  \frac{a_0}{2} + \sum_{n=1}^N a_n \cos(\frac{2 \pi n}{\lambda} z) + b_n \sin(\frac{2 \pi n}{\lambda} z) \;,  
 $$
 does not necessarily converge  to $f$ uniformly on $\R$.
 However,
   the  Cesaro means
 $$
 \sigma_n =    \frac{S_0 + \dots + S_{n-1}}{n}
  $$
  converge to $f$ uniformly on $\R$. In particular, 
  \begin{thm} \label{fejer2}  $\R[\cos(\frac{2 \pi}{\lambda} x), \sin(\frac{2 \pi}{\lambda} x)]$ is dense in 
   the $\R$-Banach algebra 
  $(\sP_{\R,\lambda}, ||~||_\R)$.
    \end{thm}
  We will show below that a suitably reformulated $p$-adic analog of Theorem~\ref{fejer2} holds true $p$-adically.
 
\begin{defn} \label{bohrdef} (Bohr's definition of u.a.p. functions) A continuous function $f:\R \to \R$ is   \emph{uniformly almost periodic} (\emph{u.a.p.} for short) if, for any $\veps >0$, there exists $\ell_\veps >0$ such that for any interval $I \subset \R$ of length $\ell_\veps$ there exists $\tau \in I$ such that 
 $$
 |f(x+\tau) - f(x)| < \veps\;\;,\;\; \forall \, x\in \R\;.
 $$
\end{defn} 
\par
It is easy to check that the set of uniformly almost periodic functions  $\R \to \R$ is a closed subalgebra $AP_\R$ of 
$(\sC_\unif^\bd(\R,\R), ||~||_\R)$ \cite[Chap. I, \S 1, Thms $4^\circ$,$5^\circ$]{Bes}. We define 
$AP_\C \subset (\sC_\unif^\bd(\R,\C), ||~||_\R)$ similarly.
\par
The following result is Bohr's ``Approximation Theorem''. We refer to \cite[I.5]{Bes} for its proof and 
for  a detailed description of the contributions of S. Bochner and H. Weyl.
\begin{thm} \label{bohrthm} $(AP_\R, ||~||_\R)$ identifies with the completion  of the normed ring 
$$(\R[\cos(\frac{2 \pi}{\lambda} x), \sin(\frac{2 \pi}{\lambda} x) \,|\, \lambda \in \R^\times], ||~||_\R)
\;.
$$
Similarly for $(AP_\C, ||~||_\R)$.
\end{thm}
We propose $p$-adic analogs of those Banach algebras and of the latter theorem. 
\subsection{Dirichlet series} \label{expvstrig} 
 Let $\C\{x\}$ be the Fr\'echet $\C$-algebra of entire functions $\C \to \C$, equipped with  the topology of uniform convergence on  compact subsets of $\C$.  
The rotation $z \mapsto i z$ transforms trigonometric series into series of exponentials 
and Bohr's definition naturally propagates  into the following
\begin{defn} \label{bochnerdef} \emph{\cite[III.2,$1^\circ$]{Bes}}. For any interval $(a,b) \subset \R$, an analytic function $f$ on the \emph{strip $(a,b) \times i \R \subset \C$} is 
 \emph{almost periodic holomorphic}  on $(a,b)$   if, for any $\veps >0$, there exists $\ell_\veps >0$ such that for any interval $I \subset \R$ of length $\ell_\veps$ there exists $\tau \in I$ such that 
 $$
 |f(x+ i \tau) - f(x)| < \veps\;\;,\;\; \forall \, x\in (a,b) \times i \R\;.
 $$
We let $\APH_\C((a,b))$ denote the $\C$-algebra  of almost periodic holomorphic functions on $(a,b)$.
\end{defn}
Notice that $\APH_\C((a,b))$ is a closed subalgebra of the Fr\'echet algebra $\cO((a,b) \times i \R)$; the induced 
Fr\'echet algebra structure is called \emph{standard}. We may equip $\APH_\C((a,b))$ with the finer Fr\'echet algebra structure  of uniform convergence on   substrips $(a',b') \times i \R$, for $a<a'<b'<b$. We informally call this topology \emph{the strip topology}.  
\endgraf
The following Polynomial Approximation Theorem  \cite[III.3,$3^\circ$]{Bes} holds. 
\begin{thm}\label{bochnerthm} $\APH_\C((a,b))$ is the Fr\'echet completion of the $\C$-polynomial algebra 
generated by the restrictions to  $(a,b) \times i \R$ of all 
continuous characters of $\R$, namely by the maps 
\beq 
e_\lambda : (a,b) \times i \R   \longrightarrow \C
\;\;,\;\;
z \longmapsto e^{\lambda z}
\eeq
for $\lambda \in \R^\times$, equipped with the strip topology.
\end{thm}
 The assignment  $(a,b) \mapsto \APH_\C((a,b))$ uniquely extends to a sheaf of Fr\'echet $\C$-algebras on $\R$. 
 \begin{defn} \label{APHdef}
 \hfill 
 \ben
 \item
 We denote by $\APH_{0,\C}$ the stalk of the sheaf $\APH_\C$ at $0$ equipped with the locally convex inductive limit topology of the system of Fr\'echet algebras $\APH_\C((-\veps,\veps))$ as $\veps \to 0^+$.
\item  We denote by $APH_\C \subset \C\{x\}$ the Fr\'echet algebra of global sections of $\APH_\C$ equipped with the strip topology. 
\een 
 \end{defn}
Notice that we have a natural injective morphism, induced by restriction of functions and the properties of the inductive limit
\beq
APH_\C \longrightarrow \APH_{0,\C} \;.
\eeq
It follows from the combined theorems  of approximation  Theorem~\ref{bohrthm} and  Theorem~\ref{bochnerthm}  that
\begin{cor} \label{bohrcor} $(AP_\C, ||~||_\R)$ identifies with the completion  of  the normed ring 
$(\APH_{0,\C}, ||~||_\R)$.
\end{cor}
\begin{rmk}\label{dirichlet} Sections of the sheaf $\APH_\C$ on open subsets of $\R$ may be viewed as generalized Dirichlet series  \cite[III.3]{Bes}. 
A $p$-adic analog on $\Q_p$ of the sheaf $\APH_\C$  of Dirichlet series on $\R$,  might be  useful
in the theory of $p$-adic L-functions. 
\end{rmk}
 \end{section}

\begin{section}{Appendix C: Numerical Calculations by M. Candilera}
\label{candcalc} 
The following calculations were performed with \emph{Mathematica}${}^\copyright$.  We computed the first coefficients of the series $\Psi_p(T) = \sum_{n=1}^\infty b_n T^n$, for $p=2$, up to the term of degree $2^5$, and for $p=3$, up to  degree $3^4$. We also evaluated the 
a few coefficients of  $\Psi_5(T)$ and  $\Psi_7(T)$. We give here tables of the $p$-adic orders of the coefficients $b_n$ for $p=2,3$. For those values of $p$, we also draw the graph of the function $n \mapsto v_p(b_n)$ and compare it with the Newton polygon of $\Psi_p$ (flipped around the $y$-axis). We confirm experimentally the calculation of the corresponding valuation polygons. 
 \subsection{Very first coefficients}
\ben
\item $p=2$
\begin{equation*}
 \begin{split}
\Psi_2(T) &= T -2  \cdot  T^2 + 2^4  \cdot  T^3 -  11 \cdot 2^5 \cdot T^4 + 7 \cdot 2^{11}\cdot  T^5  - 7 \cdot 37 \cdot  2^{12} \cdot T^6 + \\
&3 \cdot 751 \cdot  2^{16} \cdot T^7 -  301627 \cdot 2^{17} \cdot T^8 +  308621 \cdot 2^{26} \cdot  T^9 + 2^{27} \cdot T^{10} \cdot u(T) \; ,
\end{split}
\end{equation*}
   for a unit $u(T) \in \Z_{(2)}[[T]]^\times$.
\item $p=3$
\begin{equation*}
 \begin{split}
 \Psi_3(T) &=   T - 3^2  \cdot  T^3 + 3^7  \cdot T^5 - 2^2   \cdot   7 \cdot 3^{11} \cdot T^7 + \\
 2 \cdot 7  \cdot 13 \cdot 113  \cdot 3^{14} 
 \cdot  T^9  -    & 5 \cdot 89  \cdot 1249 \cdot 3^{22} \cdot T^{11} +  5 \cdot117 \cdot 217667 \cdot  3^{28}  \cdot T^{13} + \dots 
 \;.
  \end{split}
\end{equation*}
   \begin{equation*}
 \begin{split}
 \end{split}
\end{equation*}
\item $p=5$
\begin{equation*}
 \begin{split}
 \Psi_5(T) &=   T - 5^4  \cdot  T^5 + 5^{13}  \cdot T^9 - 53  \cdot   59  \cdot  5^{21}  \cdot  T^{13} + \\
   &3  \cdot  11  \cdot  97  \cdot 1123  \cdot  1699  \cdot  5^{29}  \cdot  T^{17} +   5^{37}  \cdot  T^{21}  \cdot u(T) \; ,
\end{split}
\end{equation*}
   for a unit $u(T) \in \Z_{(5)}[[T]]^\times$.
\item $p=7$
\begin{equation*} \Psi_7(T) = T - 7^6  \cdot   T^7  + 7^{19}  \cdot  T^{13} -   2  \cdot  31  \cdot  37  \cdot  359  \cdot  7^{31}  \cdot  T^{19} +  7^{43}  \cdot T^{25}  \cdot u(T) \; ,
\end{equation*}
for a unit $u(T) \in \Z_{(7)}[[T]]^\times$.
\een
%
\subsection{First 24 coefficients of $\Psi_2(t)$ and $2$-adic order of the 32 first}
\begin{figure}
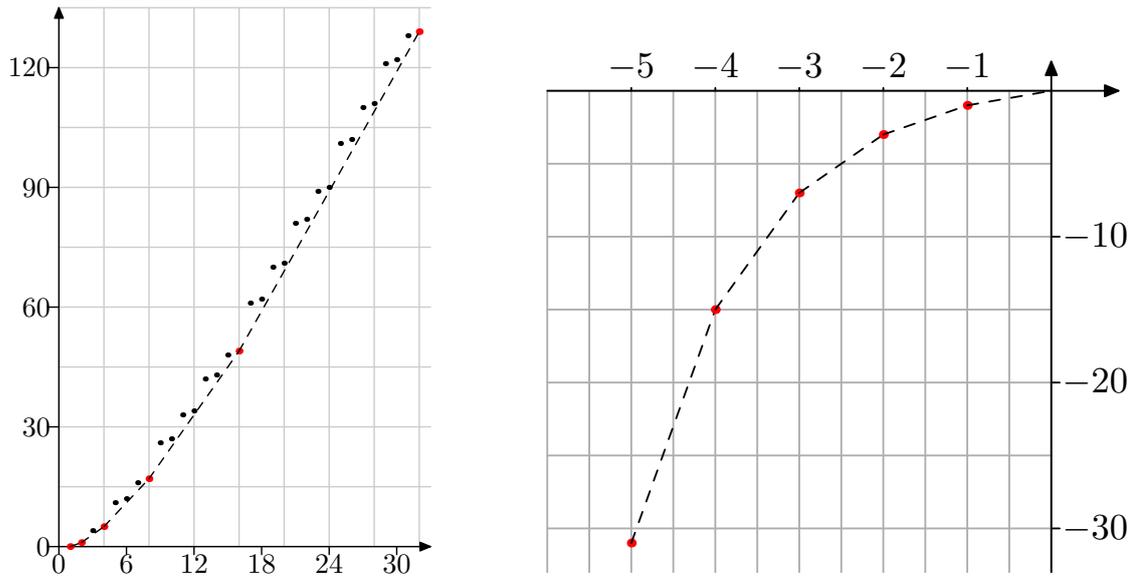

\includegraphics[width=.40\textwidth]{poligono.90}
\qquad \qquad
\includegraphics[width=.55\textwidth]{poligono.92}
   \caption{Newton and valuation polygons of $\Psi_2$.}
\end{figure}
\medskip
\begin{center}
        \begin{tabular}{|c|c||c|c||c|c||c|c|}
            \hline 
             \multicolumn{8}{|c|}{$\Psi_2(t)=\sum_{n\geq1} b_nt^n$ }\\ 
            \hline 
             $b_1$ & 0 & $b_9$    & 26 & $b_{17}$ & 61 & $b_{25}$ & 101 \\
            \hline 
             $b_2$ & 1 & $b_{10}$ & 27 & $b_{18}$ & 62 & $b_{26}$ & 102 \\
            \hline 
             $b_3$ & 4 & $b_{11}$ & 33 & $b_{19}$ & 70 & $b_{27}$ & 110 \\
            \hline 
             $b_4$ & 5 & $b_{12}$ & 34 & $b_{20}$ & 71 & $b_{28}$ & 111 \\
            \hline 
             $b_5$ & 11 & $b_{13}$ & 42 & $b_{21}$ & 81 & $b_{29}$ & 121 \\
            \hline 
             $b_6$ & 12 & $b_{14}$ & 43 & $b_{22}$ & 82 & $b_{30}$ & 122 \\
            \hline 
             $b_7$ & 16 & $b_{15}$ & 48 & $b_{23}$ & 89 & $b_{31}$ & 128 \\
            \hline 
             $b_8$ & 17 & $b_{16}$ & 49 & $b_{24}$ & 90 & $b_{32}$ & 129 \\
            \hline
        \end{tabular}
    \end{center}
   \centerline{$2$-adic valuation of the coefficients of $\Psi_2$}
\begin{scriptsize}
\begin{align*}
b_1 &=1,\qquad b_2 = -2, \qquad  b_3 =  16 = 2^4,
       \qquad b_4 =  -352 = -2^5\cdot11 \\
b_5  &=14336 = 2^{11}\cdot7, \qquad\qquad
     b_6 =  -1060864 = -2^{12}\cdot7\cdot37 \\
b_7 &= 147652608 = 2^{16}\cdot3\cdot751 \\
b_8  &=-39534854144 = -2^{17}\cdot301627 \\
b_9  &= 20711204716544 = 2^{26}\cdot308621 \\
b_{10}  &=-21454855889485824 = -2^{27}\cdot3^2\cdot13\cdot701\cdot1949 \\
b_{11}   &= 44195700516541431808 = 2^{33}\cdot5145056699 \\
b_{12}  &=-181554407879323198423040 = -2^{34}\cdot5\cdot41\cdot2273\cdot22679509 \\
b_{13}  &= 1489469015852141109009448960 = 2^{42}\cdot5\cdot67733208918623\\
b_{14}  &=-24421319844213105128638664146944 =-2^{43}\cdot3^2\cdot8179\cdot37716952983613\\
b_{15}   &= 800530746908074643997623203521363968 = 2^{48}\cdot31\cdot71\cdot1619\cdot826201\cdot966018887 \\
b_{16}   &=-52473187457503996327647036404796036743168 
                =-2^{49}\cdot31\cdot397\cdot13687\cdot2882489\cdot191972726039  \\
b_{17}   &=6878395240848057051122842718175351390427152384  
                = 2^{61}\cdot3\cdot47\cdot59\cdot919\cdot24709\cdot15791216459521333  \\
b_{18}  &=-1803212578568825704559863338710346864852507172012032 \\
                &= 2^{62}\cdot3^2\cdot19\cdot97\cdot173\cdot1665967\cdot581220517\cdot140723269997  \\
b_{19}   &=945424354393817092018179744741353462710753588534117924864  \\
                &= 2^{70}\cdot7^2\cdot23\cdot15973\cdot44485316159805664956515547941  \\
b_{20}  &=-991360632780906301560343330625129510790528483073480047449866240 \\
                &= 2^{71}\cdot5\cdot167\cdot14503\cdot15445577653440901\cdot2244675152281633901  \\
b_{21}  &=2079045830009718214618472297232655379089817022368004517660824096997376 \\
                &= 2^{81}\cdot109\cdot23549\cdot167442376921\cdot2000645152343730624200879183  \\
b_{22}   &=-8720175189463740580963423057535032711261236371520206719551905031269050744832 \\
                &= 2^{82}\cdot47\cdot1867\cdot105323\cdot2119591\cdot80618233393589\cdot1141865166972250409671  \\
b_{23}  &= 73150235997673008411264495083486904164758556563477195586370441676376428384144588800 \\
                &= 2^{89}\cdot3^3\cdot5^2\cdot175082340917111384848376265817809832605816887352831773  \\
b_{24}   &=-1227258187586069935509530355473988020883482157853428276444146736521211077001846045664083968 \\
                &= 2^{90}\cdot54617\cdot76121647308197\cdot238451637287968840726339672350427699951944293  \\
\end{align*}
\end{scriptsize}

\subsection{$3$-adic values of the first $81$ coefficients of $\Psi_3(T)$}
\begin{figure}
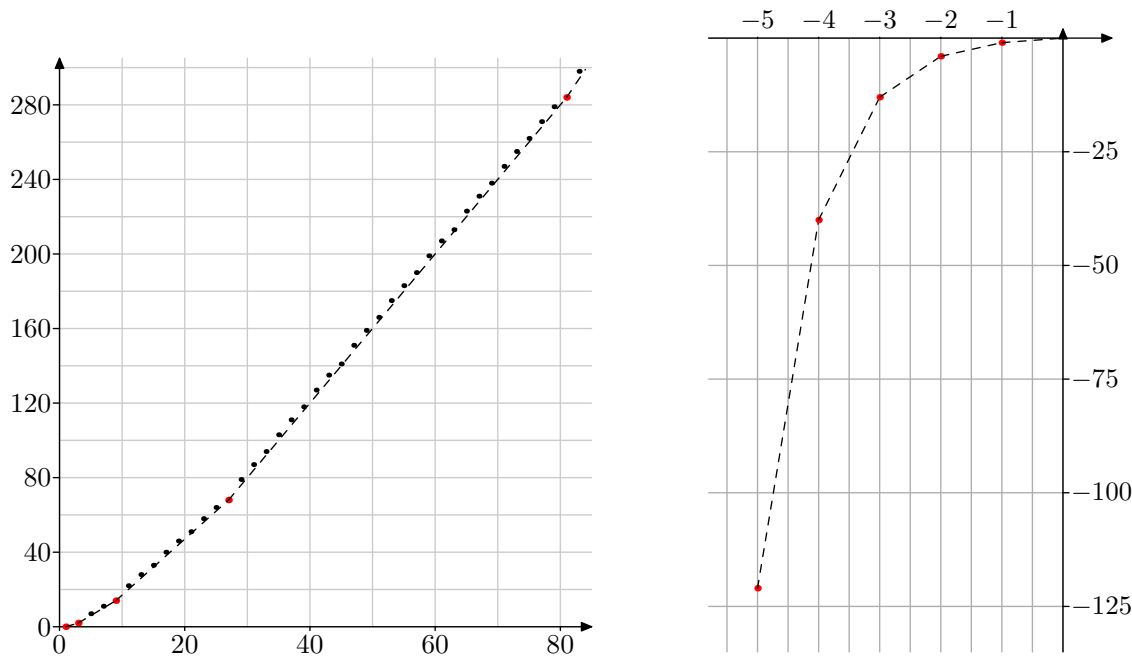

\includegraphics[width=.55\textwidth]{poligono.6}
\qquad \qquad
\includegraphics[width=.40\textwidth]{poligono.60}
\caption{The Newton and valuation polygons of $\Psi_3$.}
\end{figure}

$$
\begin{array}{|c|c|c|c|c|c|c|c|}
\hline
  b_3   & 2   & b_{23} &  58 & b_{43}  & 135 & b_{63} & 213
\\ \hline
  b_5   & 7   & b_{25} &  64 & b_{45}  & 141 & b_{65} & 223
\\ \hline
  b_7   & 11  & b_{27} &  68 & b_{47}  & 151 & b_{67} & 231
\\ \hline
b_9    & {14} & b_{29} &  79 & b_{49}  & 159 & b_{69} & 238
\\ \hline
b_{11} & {22} & b_{31} &  87 & b_{51}  & 166 & b_{71} & 247
\\ \hline
b_{13} & {28} & b_{33} &  94 & b_{53}  & 175 & b_{73} & 255
\\ \hline
b_{15} & {33} & b_{35} & 103 & b_{55}  & 183 & b_{75} & 262
\\ \hline
b_{17} &{40}  & b_{37} & 111 & b_{57}  & 190 & b_{77} & 271
\\ \hline
b_{19} &{46}  & b_{39} & 118 & b_{59}  & 199 & b_{79} & 279
\\ \hline
b_{21} &{51}  & b_{41} & 127 & b_{61}  & 207 & b_{81} & 284
\\ \hline
\end{array}
$$
\medskip
\centerline{$3$-adic valuation of the coefficients of $\Psi_3$}

\end{section}

\end{document}